\documentclass[a4paper,10pt]{article}
\usepackage[left=3cm,right=3cm]{geometry}

\usepackage[T1]{fontenc}
\usepackage[utf8]{inputenc}
\usepackage[english]{babel}
\usepackage{hyperref}
\hypersetup{
    colorlinks=true,
    linkcolor=blue,
    filecolor=magenta,
    urlcolor=cyan,
    pdfpagemode=FullScreen,
    }

\usepackage{amsthm}
\usepackage{amsmath}
\usepackage{amssymb}
\usepackage{braket,amsfonts}
\usepackage{bm} 
\usepackage{accents}
\usepackage{array}

\newtheorem{theorem}{Theorem}

\newtheorem{proposition}{Proposition}
\newtheorem{assumption}{Assumption}

\usepackage{algpseudocode}
\usepackage{algorithm}

\usepackage{authblk}

\usepackage{url} 
\usepackage{multirow}
\usepackage{setspace}
\usepackage{subcaption}

\usepackage{pgfplots}

\usepackage{amsopn}

\usepackage{xspace}
\usepackage{bold-extra}
\usepackage[most]{tcolorbox}

\colorlet{texcscolor}{blue!50!black}
\colorlet{texemcolor}{red!70!black}
\colorlet{texpreamble}{red!70!black}
\colorlet{codebackground}{black!25!white!25}

\patchcmd\newpage{\vfil}{}{}{}

\newcommand{\weight}{\bm{\theta}}
\newcommand{\itn}{{i=1}^N}

\newcommand{\vek}[1]{#1}

\newcommand{\vq}{\vek{q}}

\newcommand{\vy}{\vek{y}}

\newcommand{\mat}[1]{\bm{#1}}

\newcommand{\rvQ}{Q}
\newcommand{\rvX}{X}
\newcommand{\rvXi}{\varXi}
\newcommand{\rvY}{Y}

\newcommand{\eventalgebra}{\mathfrak{A}}

\newcommand{\eventset}{\varOmega}
\newcommand{\probability}{\mathbb{P}}
\newcommand{\probspace}{(\eventset, \eventalgebra,\probability )}

\newcommand{\expectation}[1]{\operatorname{E}\left[ #1\right]}
\newcommand{\variance}[1]{\operatorname{Var}\left[ #1\right]}
\newcommand{\cov}[1]{\textnormal{Cov}\left[ #1\right]}

\newcommand{\sR}{\mathbb{R}}

\newcommand{\sN}{\mathbb{N}}

\newcommand{\df}{{\textnormal{DF}}}

\newcommand{\dint}{\mathrm{d}}

\renewcommand{\df}{\pi}
\newcommand{\dfQ}{\df_{\rvQ}}
\newcommand{\dfXi}{\df_{\rvXi}}

\newcommand{\dimq}{n}
\newcommand{\dimy}{m}
\newcommand{\dimd}{\delta}
\newcommand{\vd}{\vek{d}}

\title{Scalable method for  {B}ayesian experimental design without integrating over posterior distribution}

\author[1\,*]{Vinh Hoang}
\author[2]{Luis Espath}
\author[3]{Sebastian Krumscheid}
\author[1,4]{Ra\'ul~Tempone}

\affil[1]{Chair of Mathematics for Uncertainty Quantification, RWTH Aachen University Germany.
	}
\affil[2]{Faculty of Science, University of Nottingham, United Kingdom.}
\affil[3]{Steinbuch Center for Computing and Institute for Applied and Numerical Mathematics, Karlsruhe Institute of Technology, 76131 Karlsruhe, Germany.}
\affil[4]{Computer, Electrical and  Mathematical Sciences and Engineering, KAUST,
and  Alexander von Humboldt professor in Mathematics of Uncertainty Quantification, RWTH Aachen University.}
\affil[*]{Corresponding author: hoang.tr.vinh@gmail.com}

\begin{document}
\maketitle


\begin{abstract}
We address the computational efficiency in solving the A-optimal Bayesian design of experiments problems for which the observational map is based on partial differential equations and, consequently, is computationally expensive to evaluate.
A-optimality is a widely used and easy-to-interpret criterion for Bayesian experimental design. 
This criterion seeks the optimal experimental design by minimizing the expected conditional variance, which is also known as the expected posterior variance.
This study presents a novel likelihood-free approach to the A-optimal experimental design that does not require sampling or integrating the Bayesian posterior distribution.
Our proposed approach is developed based on two properties of the conditional expectation: the law of total variance and the property of orthogonal projection.
The expected conditional variance is obtained via the variance of the conditional expectation using the law of total variance, and we take advantage of the orthogonal projection property to approximate the conditional expectation.
We derive an asymptotic error estimation for the proposed estimator of the expected conditional variance and show that the intractability of the posterior distribution does not affect the performance of our approach.
We use an artificial neural network (ANN) to approximate the nonlinear conditional expectation in the implementation of our method.
We then extend our approach for dealing with  the case that the domain of experimental design parameters is continuous by integrating the training process of the ANN into minimizing the expected conditional variance.
Specifically, we propose a nonlocal approximation of the conditional expectation and apply transfer learning to reduce the number of evaluations of the observation model.
Through numerical experiments, we demonstrate that our method greatly reduces the number of observation model evaluations compared with widely used importance sampling-based approaches. 
This reduction is crucial, considering the high computational cost of the observational models. 
Code is available at \href{https://github.com/vinh-tr-hoang/DOEviaPACE}{https://github.com/vinh-tr-hoang/DOEviaPACE}.

\end{abstract}

\tableofcontents


\section{Introduction}\label{sec:intro}

The design of experiments (DOE) aims to systematically plan experiments to collect data and make accurate conclusions about a particular process or system.
 This methodology has wide-ranging applications, such as in engineering \cite{lookman2019active, ilzarbe2008practical}, pharmaceutical \cite{n2017design, singh2005optimizing}, and biological fields \cite{mead2017statistical, bullard2010evaluation}.
Fundamentally, ill-posed problems and uncertainties naturally arise in the DOE.
The Bayesian experimental design provides a general probabilistic framework for dealing with these challenges \cite{chaloner1995bayesian, Huan2013}.
There are several optimality criteria available in the Bayesian experimental design framework,
for example, the A-optimal criterion minimizes the expected conditional variance (ECV), whereas the information gain criterion maximizes the expected Kullback--Leibler divergence between the prior and posterior distributions.
This study focuses on the first criterion, \emph{i.e.}, the A-optimal DOE.
Compared with alternative optimality criteria, the A-optimality criterion is of great practical appeal due to its straightforward interpretation.
Indeed, a reduced posterior variance indicates a reduction in uncertainty.

The Bayesian experimental design generally requires estimating the expected conditional statis\-tical quantities, such as the ECV or the expected information gain (EIG).
Because this task is computationally expensive, many previous works have sought to develop efficient methods for its execution.
Alexanderian et al. \cite{Alexanderian14} proposed a scalable method for solving the A-optimal DOE problem, that was specifically tailored to scenarios in which the observational map is linear.
For a nonlinear observational map, the posterior distribution is typically intractable.
 Although the Markov chain Monte Carlo (MCMC)  algorithm and the approximate Bayesian computation approach  can be used to sample an intractable posterior distribution, these approaches are too computationally intensive for the Bayesian DOE because the problem requires sampling many different posterior distributions for each design candidate \cite{RYAN201526}.
A popular approach to alleviate the computational cost of the Bayesian DOE for nonlinear observational maps is to use a Laplace approximation of the posterior distribution \cite{Rue09, ryan16, bartuska22, bartuska23}.
In \cite{Alexanderian16}, Alexanderian et al. used the Laplace approximation approach in the context of the A-optimal DOE.
In \cite{Beck2018}, Beck et al. used this Laplace approximation approach for estimating the EIG, which was later incorporated with the multilevel method in \cite{Beck2020} and the stochastic gradient descent to find the optimal design in \cite{CARLON2020}.
For a nonlinear observation map, the approaches mentioned above estimate expected conditional quantities using a three-step process: first, they sample the posterior distribution or its approximation for  each observational sample; second, they estimate the posterior characteristic, \emph{e.g.}, the posterior variance or the Kullback--Leibler divergences with respect to (\emph{w.r.t.}) the prior density; and finally, they repeat the previous step for many observational samples to estimate the expected value of the posterior characteristics.
The first two steps are equivalent to solving multiple  inverse problems,
 each of which is computationally expensive, particularly
  for high-dimensional cases,  because the posterior distribution becomes intractable.

This study presents a novel method for estimating the ECV using the projection-based approxi\-mation of the conditional expectation (PACE),  designed to handle computationally expensive observational maps.
The relationship between the uncertainty of the quantities of interest, repre-sented by random variable (RV) $\rvQ$, and their corresponding experimental observations, denoted as RV $\rvY_{\vd}$, is assumed to be $\rvY_{\vd} = h(\rvQ, d) + \rvXi$, where $h$ is the observational map, $\rvXi$ represents measurement error, and $\vd$ denotes the design parameters. 
Using the law of total variance, the ECV of the RV $\rvQ$ given RV $\rvY_{\vd}$, denoted as $\expectation{\variance{\rvQ \mid \rvY_{\vd}}}$,  is obtained by approximating the conditional expectation (CE)  of the RV $\rvQ$ given RV $\rvY_{\vd}$, denoted as $\expectation{\rvQ \mid \rvY_{\vd}}$.
 Specifically, the ECV can be computed as the difference between the variance of $\rvQ$ and the variance of CE $\expectation{\rvQ \mid \rvY_{\vd}}$. 
To approximate $\expectation{\rvQ \mid \rvY_{\vd}}$, we utilize the orthogonality of CE and introduce the map $\phi(\rvY_{\vd})$, which minimizes the mean square error $\expectation{\lVert \rvQ - \phi(\rvY_{\vd}) \rVert_{2}^2}$ under the assumption of the finite variance of both $\rvQ$ and $\rvY_{\vd}$. 
The proposed approach prevents the need to sample or approximate the posterior distribution and eliminates the evaluation of the likelihood function.
Notably, we show that the relative mean absolute error (MAE) of the PACE-based estimator of the ECV is of the order $\mathcal{O}(N^{-2})$, where $N$ represents the number of evaluations of the observational map. Moreover,  the computational efficiency of the proposed approach remains unaffected by the intractability of the posterior.
In a related work, Hoang et al.~\cite{Hoang22} used the PACE approach to develop a  machine learning-based ensemble conditional mean filter (ML-EnCMF) for nonlinear data assimilation problems and demonstrated that the ML-EnCMF outperforms linear filters in terms of accuracy.

To implement our method, we use an artificial neural network (ANN) \cite{Goodfellow2016a, rosenblatt1958perceptron} to approxi\-mate the CE, owing to ANN's versatility and demonstrated effectiveness in handling high-dimensional regression problems.
To deal with continuous design domains, we present a novel algorithm that effectively minimizes the ECV using the stochastic gradient descent method. 
The training process of the ANN  is integrated into the algorithm employed to optimize the design parameters, thereby improving computational efficiency. 
This integration is possible because the objective functions for optimizing the ECV and training the ANN are identical. 
Moreover, we propose a nonlocal approximation of the CE and apply transfer learning to reduce the number of evaluations of the observation model.

The remainder of the paper is structured as follows. In Sec.~\ref{sec:background}, we summarize the background of the Bayesian experimental design. Section~\ref{sec:estimating_tecv_using_pace} details the PACE framework used for the A-optimal DOE and its error estimation. In Section~\ref{sec:linear_setting}, we illustrate our method and verify its error estimation using the linear-Gaussian setting of the DOE. Then, in Section~\ref{sec:continuous_domain}, we examine the continuous design domain scenario and develop a stochastic optimization algorithm to solve the A-optimal DOE. In Section~\ref{sec:eit}, we apply our approach for the optimal design of electrical impedance tomography experiments used to identify the properties of a composite material. Finally, in Section~\ref{sec:conclusion}, we conclude the paper with a summary and perspectives.

\section{Background}\label{sec:background}
In Bayesian experimental design, the uncertainty associated with the quantities of interest is modeled as an RV, denoted here as $\rvQ$ and takes value in $\sR^\dimq$.  We consider a deterministic observational map, denoted as $h$, that is parameterized by a vector of design parameters $\vd$ in a domain $\mathcal{D} \subset \sR^{\dimd}$. This map transforms each vector $\vq \in \sR^\dimq$ into a corresponding noise-free observational vector in $\sR^\dimy$. We consider the scenario in which map $h$ consists of two components: a computationally expensive numerical model denoted as $h_m$, which solves the partial differential equation that governs the experiments, and a measurement operator $h_o$. Typically, $h = h_o \circ h_m$. Assuming that the measurement error is additive, we model the observational RV $\rvY_{\vd}$ for a given vector $\vd$ of the experimental design parameters as
\begin{equation} \label{eq:measurement_operator}
	{\rvY_{\vd}}(\omega) = h (\rvQ(\omega) ,\vd) + \rvXi(\omega),
\end{equation}
where $\rvXi$ is the observational error RV, and $\omega$ denotes an outcome in the sample space $\varOmega$ of the underlying probability space $\probspace$. Further, subscript $\vd$ indicates the dependence of the observational RV $\rvY_{\vd}$ on the design
parameter vector $\vd$. The inverse problem involves updating the prior distribution of the quantities of interest, given the specific measurement data, $\vy$.

Using the Bayesian identification framework is a standard approach to solve the inverse problem owing to the mathematically well-posedness of the framework and its ability to handle uncertainty. Let us assume that the distributions of the RVs $\rvQ$, ${\rvY_{\vd}}$, and $\rvXi$  have finite variances and are absolutely continuous, \emph{i.e.}, their densities exist. For a fixed vector $\vd$, given the prior PDF $\dfQ$ of RV $\rvQ$ and the observational data $\vy$,  the Bayesian posterior PDF $\df_{\rvQ \mid {\rvY_{\vd}}} (\cdot \mid \vy)$ is given as follows:
\begin{equation}\label{eq:bayesrule}
	\df_{\rvQ \mid {\rvY_{\vd}}} (\vq \mid \vy) =
	\dfrac{\dfQ(\vq) \,  \dfXi (\vy - h(\vq, \vd)) }{\df_{{\rvY_{\vd}}}(\vy)},
\end{equation}
where
$\dfXi $ is the density function of RV $\rvXi$,
and $\df_{{\rvY_{\vd}}}(\vy)$ is the evidence given by
\begin{equation}\label{eq:evidence}
	\df_{{\rvY_{\vd}}}(\vy) = \int_{\sR^\dimq}  \dfQ(\vq) \, \dfXi(\vy - h(\vq, \vd))  {}\dint \vq.
\end{equation}

The Bayesian experimental design searches for the parameter vector $\vd$ that minimizes the posterior uncertainty. Here, we  consider the A-optimal DOE, which essentially minimizes the expected posterior variance. The posterior variance can be formulated via the Bayes' rule, Eq.~(\ref{eq:bayesrule}), as
\begin{equation}\label{eq:posterior_variance}
	\variance{\rvQ \mid  \vy; \vd} = \int_{\sR^\dimq} \vq^{\odot 2} \; \df_{\rvQ \vert {\rvY_{\vd}}} (\vq | \vy) \dint \vq
	- \left[\int_{\sR^\dimq} \vq \; \df_{\rvQ \vert {\rvY_{\vd}}} (\vq | \vy) \dint \vq \right]^{\odot 2},
\end{equation}
where $^{\odot}$ denotes the Hadamard square operator, e.g., ${\vq}^{\odot 2} = [q_1^2,\dots,q_{\dimq}^2]^{\top}$.
Modeling the measurement data as an RV, the distribution of the posterior variance is represented  by the conditional variance
$\variance{\rvQ \mid {\rvY_{\vd}}}$,
which is an $\sR^\dimq$-valued RV defined as
\begin{equation}\label{eq:conditonal_variance_bayes}
	\variance{\rvQ \mid {\rvY_{\vd}}}(\omega) \equiv \variance{\rvQ \mid {\rvY_{\vd}}(\omega); \vd}
	\bigr].
\end{equation}

The A-optimal DOE approach seeks the experimental design parameters that minimize the ECV of the RV $\rvQ$ given RV $\rvY_{\vd}$, denoted as $\expectation{\variance{\rvQ \mid  {\rvY_{\vd}}}}$.
In cases where $\dimq > 1$, implying a multidimensional scenario, the A-optimal DOE approach uses the element-wise sum of the ECV as the objective function, which is referred to as the \emph{total ECV} (tECV) in this study.
We denote $\rvQ = [\rvQ_1, ..., \rvQ_{\dimq}]^\top$ where $\rvQ_i$ is the $i$-th component of $\rvQ$.
The tECV $V$ for  a given vector $\vd \in \sR^{\dimd}$ of the experimental design parameters is given by
\begin{equation}\label{eq:totalVar}
	V(\vd)  =	\sum_{i=1}^\dimq	\expectation{ \variance{\rvQ_i \mid  {\rvY_{\vd}}} },
\end{equation}
where $\variance{\rvQ_i \mid {\rvY_{\vd}}}$ is the conditional variance of the $i$-th component of the RV $\rvQ$. 
Alternatively, the tECV can be expressed as the trace of the expected conditional covariance, represented as $V(\vd) \equiv \operatorname{tr} \bigr (\cov{\rvQ \mid \rvY_{\vd} }\bigr)$,  
with $\operatorname{Cov}$ and $\operatorname{tr}$ denoting the covariance and trace operators, respectively.

We obtain the A-optimal DOE by minimizing the tECV, \emph{i.e.}, its parameter vector $\vd_{\text{A}}$ satisfies
\begin{equation}\label{eq:AoptimalED}
	\vd_{\text{A}}  =\; \arg \; \min_{\vd} \; V(\vd).
\end{equation}
By minimizing the tECV, the A-optimal DOE approach attempts to reduce the overall variability of the posterior.

Solving the optimization problem in Eq.~\eqref{eq:AoptimalED} requires an efficient and accurate estimation of the tECV. 
A double-loop algorithm is needed for using the Monte Carlo (MC) method to estimate the tECV based on Eq.~(\ref{eq:totalVar}).
The outer loop samples the RV $\rvY_{\vd}$ , and the inner loop estimates the posterior variance for each sample of the outer loop.
Let $\{\vy^{(i)}\}_{i=1}^{N_{\text{o}}}$ be $N_{\text{o}}$ samples from the outer loop. 
Estimating the posterior variance $\variance{\rvQ \mid \vy^{(i)};\, \vd} $ for each sample $\vy^{(i)}$ involves sampling the posterior distribution $ \df_{\rvQ \vert {\rvY_{\vd}}} (\cdot \mid \vy^{(i)}) $. 
However, the sampling of an intractable posterior distribution poses a significant challenge, \emph{i.e.}, the posterior can be concentrated in considerably smaller regions compared to the prior, especially in high-dimensional problems where $\dimq, \dimy \gg 1$. 
Using the importance sampling (IS) method or the MCMC method to compute the tECV becomes computationally expensive in such cases. The formulation of the  IS double-loop estimator can be found in Appendix~\ref{appendix:ll-approach}.
Enhancing the efficiency of the IS estimator can be done using the Laplace approximation. 
However, this method assumes that the underlying distribution closely resembles a Gaussian distribution and requires that the maximum a posteriori estimator be determined.

Notably, estimating the tECV involves applying the IS method to compute the variances of $N_{\text{o}}$ different posteriors, which results in 
a computational cost of the multiplicative order $N_{\text{o}} \times N_{\text{i}}$, where $N_{\text{i}}$ is the number of samples in the inner loop. Similarly, other approaches that estimate tECV by sampling the posterior, such as the MCMC method, suffer from the same multiplicative growth of computational cost. Herein, we present a novel approach to approximate the tECV via the CE taking into account the orthogonal projection of the CE.
The PACE-based approach, as discussed in  Section~\ref{sec:estimating_tecv_using_pace}, does not require approximating or sampling the posterior distribution.

\section{Estimation of the tECV using PACE}\label{sec:estimating_tecv_using_pace}
We demonstrate that the tECV ($V(\vd)$) can be evaluated by approximating the CE rather than sampling the conditional variance. 
For a given experimental parameter vector $\vd$, the CE of RV $\rvQ$ given RV ${\rvY_{\vd}}$, which is denoted as $\expectation{\rvQ \mid {\rvY_{\vd}}}$, satisfies
\begin{equation}
\int_A \rvQ(\omega)\dint \probability (\omega) =  \int_A \expectation{\rvQ \mid \rvY_{\vd}}(\omega) \dint \probability (\omega),
\end{equation}
for every measurable set $A$ in the $\sigma$-algebra generated by the RV ${\rvY_{\vd}}$.
According to Doob--Dynkin lemma, the CE is a composition of the map $\phi_{\vd} \; : \; \sR^\dimy \rightarrow \sR^\dimq$
and the RV ${\rvY_{\vd}}$ which is given as
\begin{equation}
\expectation{\rvQ \mid {\rvY_{\vd}}} (\omega) = \phi_{\vd} ({\rvY_{\vd}}(\omega) ).
\end{equation}
The map $\phi_{\vd}$ can be defined using the posterior density as
\begin{equation}\label{eq:phi_Bayesian}
\phi_{\vd} (\vy) =  \int_{\sR^\dimq} \vq \; \df_{\rvQ \vert {\rvY_{\vd}}} (\vq , \vy) \dint \vq.
\end{equation}

Alternatively, map $\phi_{\vd}$ can be obtained from the orthogonal projection property.
For the finite-variance RVs $\rvQ$ and $\rvY_{\vd}$, the CE $\expectation{\rvQ \mid {\rvY_{\vd}}}$
is the $L_2$ orthogonal projection of RV $\rvQ$ onto the $\sigma$-algebra generated by the RV ${\rvY_{\vd}}$, which is formulated as
\begin{equation}\label{eq:cm_orthogonal_projection}
	\expectation{\rvQ \mid {\rvY_{\vd}}} = \phi_{\vd}  ({\rvY_{\vd}}), \quad \text{where} \quad \phi_{\vd}  =  \arg \min_{f
		\in \mathcal{S}(\sR^\dimy, \sR^\dimq)} \expectation{\bigl \lVert\rvQ- f \circ {\rvY_{\vd}} \bigr\rVert_{2}^2},
\end{equation}
where $\lVert\cdot\rVert_2$ is the $L_2$ norm. 
Here, $\mathcal{S}(\sR^\dimy,\sR^\dimq)$ represents the set of all functions $f:\sR^\dimy \rightarrow \sR^\dimq$ for which the variance of the RV $f({\rvY_{\vd}})$ is finite.
Appendix~\ref{appendix:orthogonal_projection} provides a proof of the orthogonal projection property.
Theoretical properties of
the CE can be found in \cite{Bobrowski2005a} and in
\cite[Chapter 4]{durrett2019probability}. 
The following proposition provides a formulation for the tECV calculated using the CE $\expectation{\rvQ \mid {\rvY_{\vd}}}$.

\begin{proposition}\label{proposition:tecv_via_conditional_expectation}
The tECV $V(\vd)$ given in Eq.~\eqref{eq:totalVar} can be expressed using
CE $\expectation{\rvQ \mid {\rvY_{\vd}}} $ as:
\begin{equation}\label{eq:theorem1}
	V(\vd) = \expectation{\left\lVert\rvQ -\expectation{\rvQ \mid {\rvY_{\vd}}}\right\rVert_{2}^2}.
\end{equation}
\end{proposition}

A proof of Proposition~\ref{proposition:tecv_via_conditional_expectation} is given in Appendix~\ref{appendix:proof_totalVar}. Using Proposition~\ref{proposition:tecv_via_conditional_expectation}, the tECV can be estimated by approximating CE $\expectation{\rvQ \mid {\rvY_{\vd}}}$.

\subsection{PACE-based MC estimator of the tECV}\label{sec:pace_tecv}
We leverage the orthogonal projection property described in Eq.~(\ref{eq:cm_orthogonal_projection}) to approximate the CE, eliminating the necessity to sample the posterior distributions as in Eq.~(\ref{eq:phi_Bayesian}).
For a nonlinear observational map $h$, the map $\phi_{\vd}$ is typically nonlinear and does not allow a closed-form solution.
In such cases, we employ the nonlinear regression method to approximate the CE, which involves three key components.
First, we select a suitable parameterized functional subspace of finite dimension $\mathcal{S}'(\sR^\dimy, \sR^\dimq) \subset \mathcal{S}(\sR^\dimy, \sR^\dimq)$. 
This subspace represents a restricted set of functions that are used to approximate the map $\phi_{\vd}$. Second, we solve the projection problem described in Eq.~(\ref{eq:cm_orthogonal_projection}) within this chosen subspace.
Last, we employ an appropriate MC estimator to estimate the MSE $\expectation{\left \lVert\rvQ- f \circ {\rvY_{\vd}} \right \rVert_{2}^2}$ for $f \in \mathcal{S}'$.

By restricting $\mathcal{S}'$ to a linear function space, we obtain a linear approximation of the map $\phi_{\vd}$, as detailed in Appendix~\ref{appendix:linear_approximation}.
However, this linear approximation of the CE is known to be biased and tends to overestimate the tECV. 
The remainder of this section is dedicated to developing a suitable method for approximating the nonlinear CE.

Our primary objective in this section is to approximate the CE given a fixed design parameter vector $\vd$.
We will expand our approach in Section~\ref{sec:continuous_domain} to encompass the neighborhood surrounding a given vector $\vd$.
Let $\weight \in \sR^{\beta}$ denote the vector containing the hyperparameters of functions $f$ in the selected subspace $\mathcal{S}'$.
We use the orthogonal projection property (Eq.~(\ref{eq:cm_orthogonal_projection})) to approximate the map $\phi_{\vd}$ using a function $f^* := f(\cdot\; ; \weight^*)$ in the subspace $\mathcal{S}'$, such that
\begin{equation}\label{eq:approxiamte_theta}
	\weight^* = \arg \; \min_{\weight} \; \expectation{\bigl\lVert  \rvQ - f({\rvY_{\vd}}; \weight)  \bigr\rVert_{2}^2}.
\end{equation}
To simplify the notation, we  introduce the MSE function $\mathcal{M}: \mathcal{S} \rightarrow \sR_{+}$, which is defined as 
\begin{equation}
	\mathcal{M}(f) \equiv \expectation{\bigl\lVert  \rvQ - f({\rvY_{\vd}})  \bigr\rVert_{2}^2}.
\end{equation}
With this notation in place, the relations between functions in $\mathcal{S}$ can be summarized as: 
$$\mathcal{M} (f) \geq \mathcal{M} (f^*) \geq \mathcal{M} (\phi_{\vd}) = V (\vd),\; \forall f\in \mathcal{S'}.$$

A straightforward method for estimating the MSE $\mathcal{M}(f)$ is to use the crude Monte Carlo (MC) estimator.
Let $D_N =  \{\left(\vq^{(i)},\vy^{(i)}\right)\}_{i=1}^N$ be an $N$-sized dataset of \emph{i.i.d.} samples of the RV pairs $(\rvQ, \rvY_{\vd})$, where $\{\vq^{(i)} \}_{i=1}^{N}$ are the \emph{i.i.d.} samples of RV $\rvQ$, and $\{\vy^{(i)} \}_{i=1}^{N}$ are the corresponding \emph{i.i.d.} samples of RV ${\rvY_{\vd}}$, which are obtained as
\begin{equation}\label{eq:samplesY}
	\{\vy^{(i)} \}_{i=1}^{N} = \left\{ \vy^{(i)} \mid \vy^{(i)}
	= h(\vq^{(i)}, \vd) + {\xi}^{(i)},\quad  \vq^{(i)} \in \{\vq^{(i)}\}_{i=1}^{N} \right\}.
\end{equation}
Here, $\{\xi^{(i)}\}_{i=1}^{N}$ are the \emph{i.i.d.} samples of $\rvXi$. 

Let $D_M = \{\left( \vq^{(i)},\vy^{(i)} \right ) \}_{i=1}^M$ be an $M$-sized dataset of \emph{i.i.d.} samples of the RV pairs $(\rvQ, \rvY_\vd)$, which is
statistically independent from $D_N$. We denote the crude MC estimator of the MSE $\mathcal{M}(f)$ using dataset
$D \in \{D_N, D_M\}$ as:
\begin{equation}\label{eq:crude}
	\widehat{\mathcal{M}}(f\mid D) = \dfrac{1}{\vert D \vert} \sum_{(\vq, \vy) \in D}
	\bigl\lVert \vq - f(\vy)\bigr\rVert _2^2,
\end{equation}
where $\vert D \vert$ is the cardinality of dataset $D$.
Owing to Proposition~\ref{proposition:tecv_via_conditional_expectation}, we estimate the tECV using the following \emph{PACE-based MC estimator} $\widehat{V}_{\vd}(D_N, D_M)$:
\begin{subequations}
	\begin{align}
		\widehat{V}_{\vd}(D_N, D_M) &:= \widehat{\mathcal{M}}(f(\cdot;\weight_{D_N}) \mid D_M)\quad \label{eq:mc}\\
		\text{where} \quad \weight_{D_N} &= \arg \; \min_{\weight}
		\;\widehat{\mathcal{M}}(f(\cdot;\weight) \mid D_N), \label{eq:mc_mse}
	\end{align}
\end{subequations}
which requires a numerical solution of the optimization problem stated in Eq.~(\ref{eq:mc_mse}).
We observe that unlike the double-loop IS estimator, which incurs a computational cost proportional to the product $(N_{\text{o}} \times N_{\text{i}})$ as explained in Section~\ref{sec:background}, the PACE-based approach exhibits a linear cost proportional to the sizes of datasets $D_N$ and $D_M$.

Although the PACE-based approach requires approximating the CE, it completely alleviates the need to sample different posterior densities.
This allows us to bypass the complication associated with posterior intractability.
Using the regressor, we obtain approximate values of the CE samples with minimal computing effort.
In the following subsections, we analyze the statistical error of the PACE-based MC estimator of the tECV. Moreover, we propose a control-variate version for this PACE-based MC estimator.

\subsection{Estimation of errors}\label{sec:error_estimation}
In this subsection, we aim to analyze the MAE of the estimator $\widehat{V}_{\vd}(D_N, D_M)$.
Let $\epsilon_{S'}$, $\epsilon_{\text{opt}}$, and $\epsilon_{\text{MC}}$ denote the approximation error due to the choice of the
subspace $\mathcal{S}'$, the optimization error for the numerical solution to the minimization problem
stated in Eq.~(\ref{eq:mc_mse}) with the
finite-size dataset $D_M$,
and the estimating error of the MC estimator due to the finite size
of the dataset $D_N$, respectively, \emph{i.e.},
\begin{align}\label{eq:error_definition}
	\epsilon_{S'} &:=\expectation{\bigl \vert \mathcal{M}(f^*)
		- \mathcal{M}(\phi_{\vd})  \bigr\vert},\\
	\epsilon_{\text{opt}} &:=\expectation{\bigl \lvert \mathcal{M}(f(\cdot;\weight_{D_N}))
		- \mathcal{M}(f^*)  \bigr \rvert},\\
	\epsilon_{\text{MC}} &:=\expectation{\bigl \lvert \widehat{\mathcal{M}}(f(\cdot;\weight_{D_N}) \mid D_M)
		- \mathcal{M}(f(\cdot;\weight_{D_N})) \bigr \rvert}. 
\end{align}

The MAE between the estimator $\widehat{V}_{\vd}(D_N, D_M)$ and $V_{d}$ is bounded using the triangular inequality as follows:
\begin{equation}\label{eq:error_bound}
	\begin{aligned}
		\expectation{\bigl \vert \widehat{V}_{\vd}(D_N, D_M)  -
			V_{\vd} \bigr\vert} = \;&
		\expectation{\bigl \vert \widehat{\mathcal{M}}(f(\cdot;\weight_{D_N}) \mid D_M)
			- \mathcal{M}(\phi_{\vd}) \bigr\vert}\\
		\leq\; &
		\expectation{\bigl \vert \widehat{\mathcal{M}}(f(\cdot;\weight_{D_N}) \mid D_M)
			- \mathcal{M}(f(\cdot;\weight_{D_N})) \bigr\vert}\\
		&+
		\expectation{\bigl \vert \mathcal{M}(f(\cdot;\weight_{D_N}))
			- \mathcal{M}(f^*)  \bigr\vert}\\
		&+
		\expectation{\bigl \vert \mathcal{M}(f^*)
			- \mathcal{M}(\phi_{\vd})  \bigr\vert} \\
		= \;&  \epsilon_{\text{MC}} +  \epsilon_{\text{opt}} + \epsilon_{S'}.
	\end{aligned}
\end{equation}

Aiming at deriving an asymptotic error estimator that captures the error evolution \emph{w.r.t.} the size of the datasets $D_N$ and $D_M$, we make use of the following assumptions:

\begin{assumption}\label{assumption1}
	$\mathcal{S'}$ is a convex set  that contains constant functions.
\end{assumption}

\begin{assumption}\label{assumption2}
	\begin{equation}
		\variance {\left \lVert \rvQ - f^*(\rvY_{\vd})\right \rVert_{2}^2} =
		\mathcal{O}\bigl(2 \expectation { \left \lVert \rvQ - f^*({\rvY_{\vd}})\right \rVert_{2}^2}^2 \bigr)\; < \infty.
	\end{equation}
\end{assumption}

\begin{assumption}\label{assumption3}
	\begin{equation}
		\epsilon_{\text{opt}}
		= \mathcal{O} \left(\expectation{
			\left \lvert \widehat{\mathcal{M}}(f^*\vert D_N) - \mathcal{M}(f^*) \right\rvert}\right).
	\end{equation}
\end{assumption}

The first assumption is a common constraint imposed on the choice of the subspace $\mathcal{S'}$.
This assumption implies that $\expectation{(\rvQ - f^{*}({\rvY_{\vd}}))^\top g({\rvY_{\vd}}))}=0$ for every $g \in \mathcal{S'}$.
By choosing $g$ as a constant function in this expression, we deduce that the mean of the RV $\rvQ - f^{*}({\rvY_{\vd}})$ is zero.
We justify the second assumption by noting that for every zero-mean Gaussian RV $X$, we have $\variance{X^2}= \expectation{X^4} - (\expectation{X^2})^2 = 2(\expectation{X^2})^2$.
The third assumption asserts that the error of the optimization problem under the finite-sized dataset $D_N$ is on the same order as the statistical error of the MC estimator $\widehat{\mathcal{M}}(f^*\vert D_N)$.

\begin{proposition}[Error estimation]\label{proposition:error_estimation}
	Assuming that the RVs $\rvQ$ and ${\rvY_{\vd}}$ have finite variance,
	and that the Assumptions~\ref{assumption1}, \ref{assumption2}, and \ref{assumption3} hold,
	the MAE of the PACE-based MC estimator $\widehat{V}_{\vd}(D_N, D_M)$ satisfies
	\begin{equation}
		\expectation{\bigl \vert \widehat{V}_{\vd}(D_N, D_M)  -
			V_{\vd} \bigr\vert} 
		=
		\mathcal{O}\Biggl( \Bigl(\dfrac{2}{\sqrt{\pi N}} + \dfrac{2}{\sqrt{\pi M}}\Bigr)
		\expectation { \left \lVert \rvQ - f^*(\rvY_{\vd})\right \rVert_{2}^2}\Biggr)
		+ \epsilon_{S'}.
	\end{equation}
	as ${N, M\ \gg 1}$.
\end{proposition}

If we further assume the approximation error  $\epsilon_{S'} =0$, then
\begin{equation}
	\begin{aligned}
		\expectation{\bigl \lvert \widehat{V}_{\vd}(D_N, D_M)  - V_{\vd} \bigr\rvert} 	
		& = \mathcal{O}\Biggl( \Bigl(\dfrac{2}{\sqrt{\pi N}} + \dfrac{2}{\sqrt{\pi M}}\Bigr)
		\expectation { \left \lVert \rvQ - f^*(\rvY_{\vd})\right \rVert_{2}^2}\Biggr) \\
		& = \mathcal{O}\Biggl( \Bigl(\dfrac{2}{\sqrt{\pi N}} + \dfrac{2}{\sqrt{\pi M}}\Bigr)
		\expectation { \left \lVert \rvQ - \phi_{\vd}(Y)\right \rVert_{2}^2}\Biggr) \\
		& = \mathcal{O} \left(\dfrac{2}{\sqrt{\pi N}} + \dfrac{2}{\sqrt{\pi M}}\right) V_{\vd}
	\end{aligned}
\end{equation}
To evaluate the accuracy of an estimator $\widetilde{V}_{\vd}$,  we use the relative MAE defined as
\begin{equation}\label{eq:numMAE}
	\text{relMAE}\; (\widetilde{V}_{\vd}) \equiv \dfrac{\expectation{\vert \widetilde{V}_{\vd}-V_{\vd} \vert}}{V_{\vd}}.
\end{equation}
When $\epsilon_{S'} =0$, the relative MAE of the estimator $\widehat{V}_{\vd}(D_N, D_M)$ is given by
\begin{equation}\label{eq:relMAE_theorectical}
	\begin{aligned}
		\text{relMAE}\; (\widehat{V}_{\vd}) &\equiv \dfrac{\expectation{\bigl \vert \widehat{V}_{\vd}(D_N, D_M)  -
				V_{\vd} \bigr\vert}}{V_{\vd}} \\
		& = \mathcal{O}\Bigl(\dfrac{2}{\sqrt{\pi N}} + \dfrac{2}{\sqrt{\pi M}}\Bigr),
	\end{aligned}
\end{equation}
for $N, M \gg 1$.
When the upper bounds of the errors $\epsilon_{S'}$, $\epsilon_{\text{opt}}$, and $\epsilon_{\text{MC}}$ are available, it is possible to derive more precise bound for the MAE of the PACE-based MC estimator. However, those upper bounds are problem-dependent and beyond the scope of this study.

\subsection{Reduced variance estimator}\label{sec:reduced_variance_estimator}
We often model the measurement error RV $\rvXi$ using simple parameterized distributions such as the Gaussian distribution, Poisson distribution, or binomial distribution. 
Consequently, sampling the RV $\rvXi$ is computationally inexpensive.
Considering this observation, we augment datasets $D_N$ and $D_M$ to reduce the statistical error of the PACE-based MC estimators.
The augmented datasets are obtained as follows:
\begin{align}\label{eq:data_augmented_sets}
	D_N^{\text{rv}} &= \left\{ \left(\vq^{(i)},\vy^{(i, j)} \right ) \mid  \vy^{(i, j)}
	= h(\vq^{(i)}, \vd) + {\xi}^{(i, j)}\right\}_{i=1, \dots, N,\; j=1,\dots,a},\\
	D_M^{\text{rv}} &= \left\{ \left ( \vq^{(i)},\vy^{(i, j)} \right )  \mid \vy^{(i, j)}
	= h(\vq^{(i)}, \vd) + {\xi}^{(i, j)}\right\}_{i=1, \dots, M,\; j=1,\dots,a},
\end{align}
where $\vq^{(i)}$ and  $\xi^{(i, j)}$ are the \emph{i.i.d.} samples of RVs $\rvQ$ and $\rvXi$, respectively, and $a \in \sN_{+}$ represents the augmentation multiplier.
By substituting the augmented datasets $D_N^{\text{rv}}$ and $D_M^{\text{rv}}$ for $D_N$ and $D_M$, respectively, in Eq.~(\ref{eq:crude}),
we obtain the reduced variance estimators of the MSEs $\widehat{\mathcal{M}} (f \mid D_N)$ and $\widehat{\mathcal{M}} (f \mid D_M)$  as follows:
\begin{equation}\label{eq:mc_vd}
	\begin{aligned}
		\widehat{\mathcal{M}}^{\text{rv}} (f \mid D_{N}) &=
		\dfrac{1}{\vert D_N \vert \times a } \sum_{(\vq, \vy) \in D^{\text{rv}}_N}
		\left \lVert \vq - f(\vy) \right \rVert_{2}^2\\
		\widehat{\mathcal{M}}^{\text{rv}} (f \mid D_{M}) &=
		\dfrac{1}{\vert D_M \vert \times a } \sum_{(\vq, \vy) \in D^{\text{rv}}_M}
		\left \lVert \vq - f(\vy) \right \rVert_{2}^2.
	\end{aligned}
\end{equation}
Appendix~\ref{appendix:data_augmenting} explains
the reduction in the statistical error obtained with the use of the
estimator $\widehat{\mathcal{M}}^{\text{rv}}$ relative to the crude one $\widehat{\mathcal{M}}$ .

\section{Numerical experiment: Linear-Gaussian setting}
\label{sec:linear_setting}
In this section, we present the results of numerical experiments to demonstrate the error estimation outlined in Section~\ref{sec:error_estimation}.
We focus on a linear-Gaussian scenario, where both distributions of the prior and the measurement error are Gaussian and the observational map is linear.
This setting is advantageous for error analysis because it allows for an analytical solution to the tECV.
We will expand our numerical experiments to a broader nonlinear setting in Section~\ref{sec:eit}.

To assess the effectiveness of the PACE-based approach, we conduct numerical experiments in two different scenarios. In Section~\ref{sec:1d}, we apply our method to a one-dimensional setting, \emph{i.e.}, $\vq \in \sR$, and in Section~\ref{sec:nd}, we evaluate our approach in a high-dimensional scenario.
Furthermore, we compare the performance of the PACE-based approach with that of the IS-based approach, by examining the effect of two conditions: (1) increasing the dimension of the inferred parameter vector  $\vq$,  and (2) reducing the measurement error variance.
These factors significantly exacerbate the intractable property of the posterior distribution.

\subsection{One-dimensional case}\label{sec:1d}
In this section, we consider the following setting:
\begin{subequations}
	\begin{align}
		&h(\rvQ, \vd)= \dfrac{\rvQ}{(\vd-0.5)^2+1},\quad \vd\in [0,1], \\
		&\rvQ\sim \mathcal{N}(0,\sigma_q^2),\; \text{with} \quad \sigma_q= 2,\\
		&\rvXi \sim \mathcal{N}(0, \sigma_{\xi}^2),
	\end{align}
\end{subequations}
where $\mathcal{N} (0, \sigma^2)$ is a Gaussian distribution with a mean of zero and variance $\sigma^2$. 
We analyze two cases of variances in measurement error:
$\sigma^2_{\xi} = 0.01^2$ and  $\sigma^2_{\xi} = 0.001^2$.
Given that the observational map is linear in terms of $\vq$ and both the prior and measurement error distributions are Gaussian, the CE map $\phi_{\vd}$ defined in Eq.~(\ref{eq:cm_orthogonal_projection}) is linear and possesses a closed-form expression. Thus, we obtain
\begin{equation}
	\phi_{\vd}(\vy)= \frac{a\sigma_q^2}{a^2\sigma_q^2+\sigma_\xi^2}\, \vy.
	\; \text{where} \;a = \dfrac{1}{(d-0.5)^2+1}.
\end{equation}
Appendix~\ref{appendix:linear_approximation} provides a detailed description of this linear approximation.

To implement our method in this setting, we employ linear regression to empirically approxi-mate the CE by solving Eq.~(\ref{eq:cm_orthogonal_projection}) within the subspace of linear functions $S'$, given the sample dataset $D_N = \{(\vq^{(i)}, \vy^{(i)})\}_{i=1}^N$.
For detailed information, please refer to Appendix~\ref{appendix:linear_approximation_empirical}.
We then utilize the obtained CE approximation to estimate the tECV using Eq~(\ref{eq:mc_vd}).
To assess the accuracy of our estimation, we compute the empirical relative MAE metric defined in Eq.(\ref{eq:numMAE}).
In particular, we aim to compare the relative MAE with the estimation in Eq.(\ref{eq:relMAE_theorectical}),
considering the absence of bias error $\epsilon_{S'}$ in the linear-Gaussian setting.
Additionally, we implement the IS-based approach described in Appendix~\ref{appendix:ll-approach} to estimate the tECV and to evaluate its relative MAE.

Based on the result of Proposition~\ref{proposition:error_estimation}, setting $N\equiv M$ is a straightforward decision for the PACE-based approach. 
In the IS-based approach, the ECV is estimated using a double-loop MC algorithm, which is detailed in Appendix~\ref{appendix:ll-approach}.
Because the posterior variance remains invariant \emph{w.r.t} $\vy$ in the linear-Gaussian setting, we set the number of samples used by the outer loop to $N_{\text{o}}=1$. 
In the PACE-based approach, the total number of samples is $N+M$, while for the IS-based approach, it is $N_{\text{i}}+1$, where $N_{\text{i}}$ represents the number of processed by the inner loop.
Fig.~\ref{fig:ce_1d} shows the relative MAE for $d = 0.5$ for two different measurement variances, \emph{i.e.}, $\sigma_{\xi}^2= 0.01^2$ and $\sigma_{\xi}^2= 0.001^2$.
Fig.~\ref{fig:ce_1d_many_d} showcases the tECV for different values of $\vd$ and  for $\sigma_{\xi}^2= 0.01^2$.
To estimate the relative MAE, we perform 1000 statistically independent simulations.

\begin{figure}[!ht]
	\centering
	\begin{subfigure}[b]{0.45\textwidth}
		\includegraphics[scale=0.4]{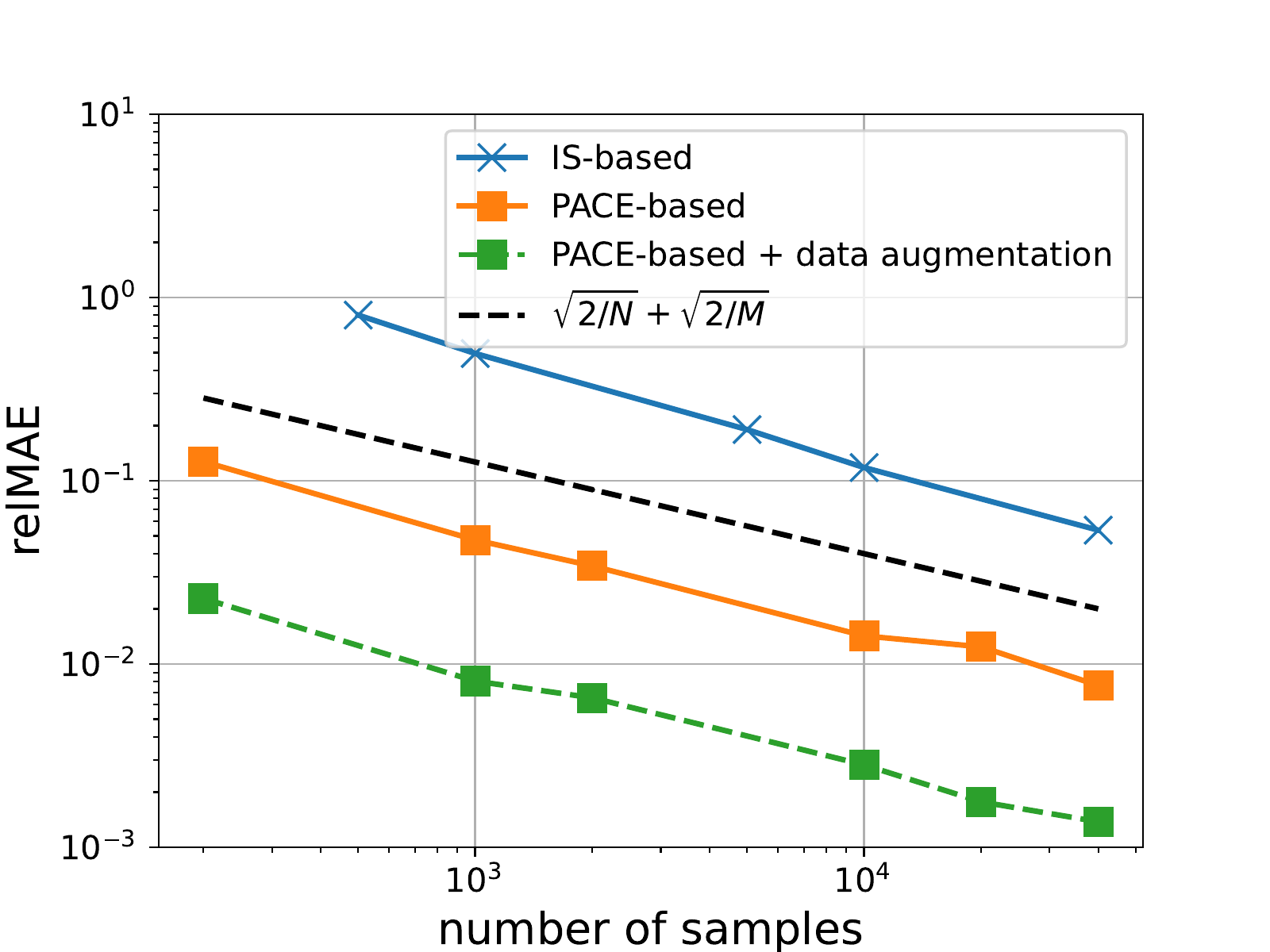}
		
		\caption{$~$}
	\end{subfigure}
	\begin{subfigure}[b]{0.45\textwidth}
		\includegraphics[scale=0.4]{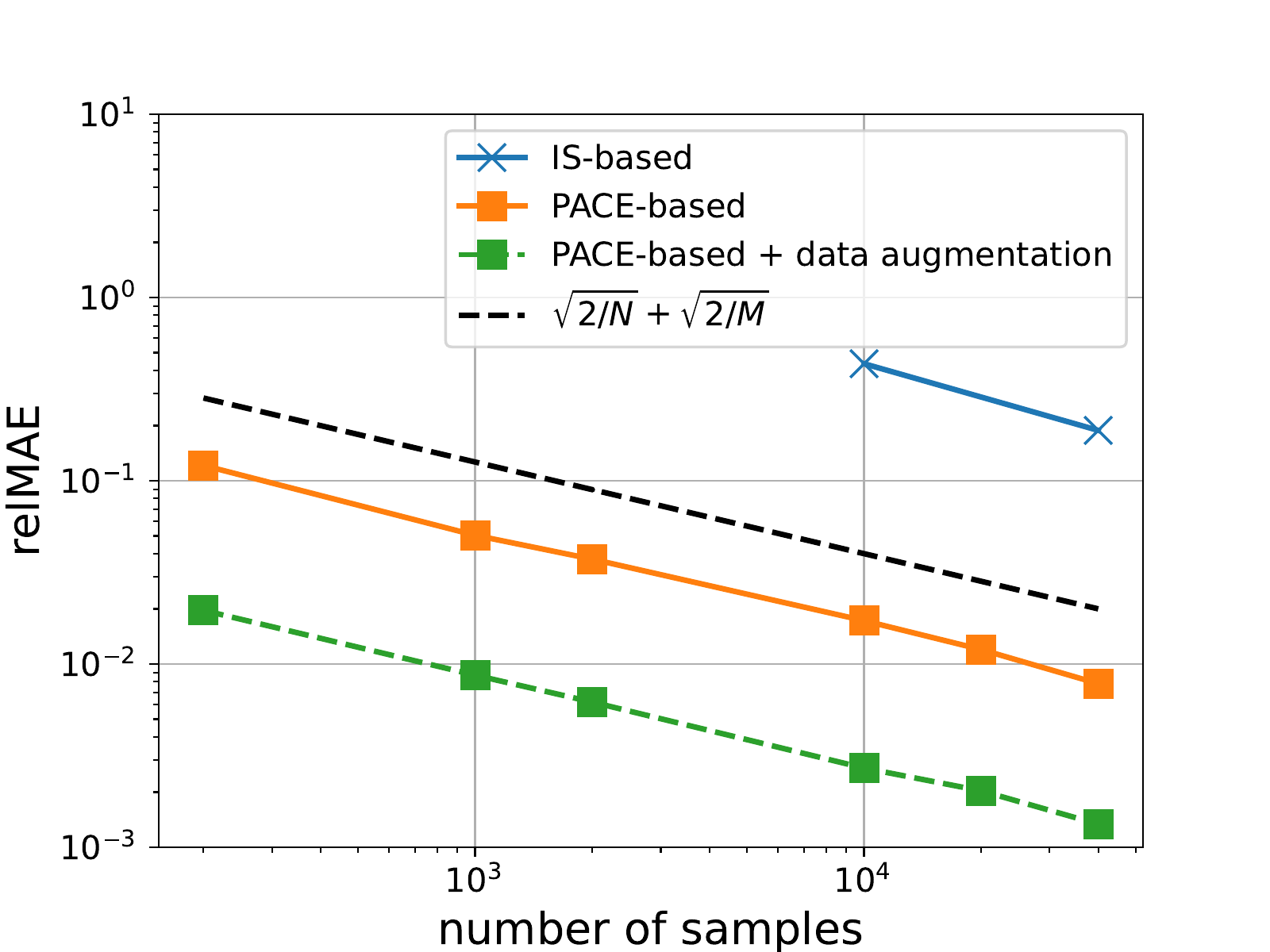}
		\caption{$~$}
	\end{subfigure}
	
	\caption{Comparison of the PACE- and IS-based approaches for estimating the tECV
		for $d = 0.5$:  (a) $\rvXi\sim \mathcal{N}(0,0.01^2)$ and
		(b) $\rvXi\sim \mathcal{N}(0,0.001^2)$.}
	
	\label{fig:ce_1d}
\end{figure}

As can be observed in Fig.~\ref{fig:ce_1d}, the PACE-based method requires considerably fewer samples than the IS-based approach for achieving a similar relative MAE value.
For example, for $\sigma_{\xi}^2= 0.01^2$ without applying data augmentation, the required number of samples is reduced by approximately 30 fold, and for $\sigma_{\xi}^2= 0.001^2$, it is reduced by about 200 fold.
When data augmentation is used, our approach further reduces the number of samples by nearly two additional orders of magnitude.

Overall, the evolution of the empirical relative MAE confirms the theoretical error estimation obtained in
Proposition~\ref{proposition:error_estimation}, \emph{i.e.},
$\mathcal{O}\bigl(\frac{2}{\sqrt{\pi N}} +\frac{2}{\sqrt{\pi M}} \bigr)$.
Particularly, we have $\;\variance {\left(\rvQ - f^*(\rvY_{\vd})\right)^2} =
2 \expectation {\left( \rvQ - f^*({\rvY_{\vd}}) \right)^2}^2$, which validates Assumption~\ref{assumption2}.
As the CE in this example is a linear function, $\epsilon_{opt}$ is much smaller than the estimated order $\Bigl (\expectation{
	\left \lvert \widehat{\mathcal{M}}(f^*\vert D_N) - \mathcal{M}(f^*) \right\lvert}\Bigr )$ that is assumed in the Assumption~\ref{assumption3}.
Consequently, the obtained relative MAE is consistently smaller than the theoretical error estimation of $\bigl (\frac{2}{\sqrt{\pi N}} +\frac{2}{\sqrt{\pi M}} \bigr)$.

We also observe that the relative MAE in the PACE-based approach is \emph{invariant} \emph{w.r.t.} the measurement error variance.
In contrast, the IS-based method requires a considerably larger number of samples for a smaller measurement error variance. Particularly, a reduction in the measurement error variance decreases the evidence terms of Bayes' formula. Consequently, the statistical error inherent in the IS-based approach becomes amplified.

\begin{figure}[!ht]
	\centering
	\begin{subfigure}[b]{0.49\textwidth}
		\includegraphics[scale=0.4]{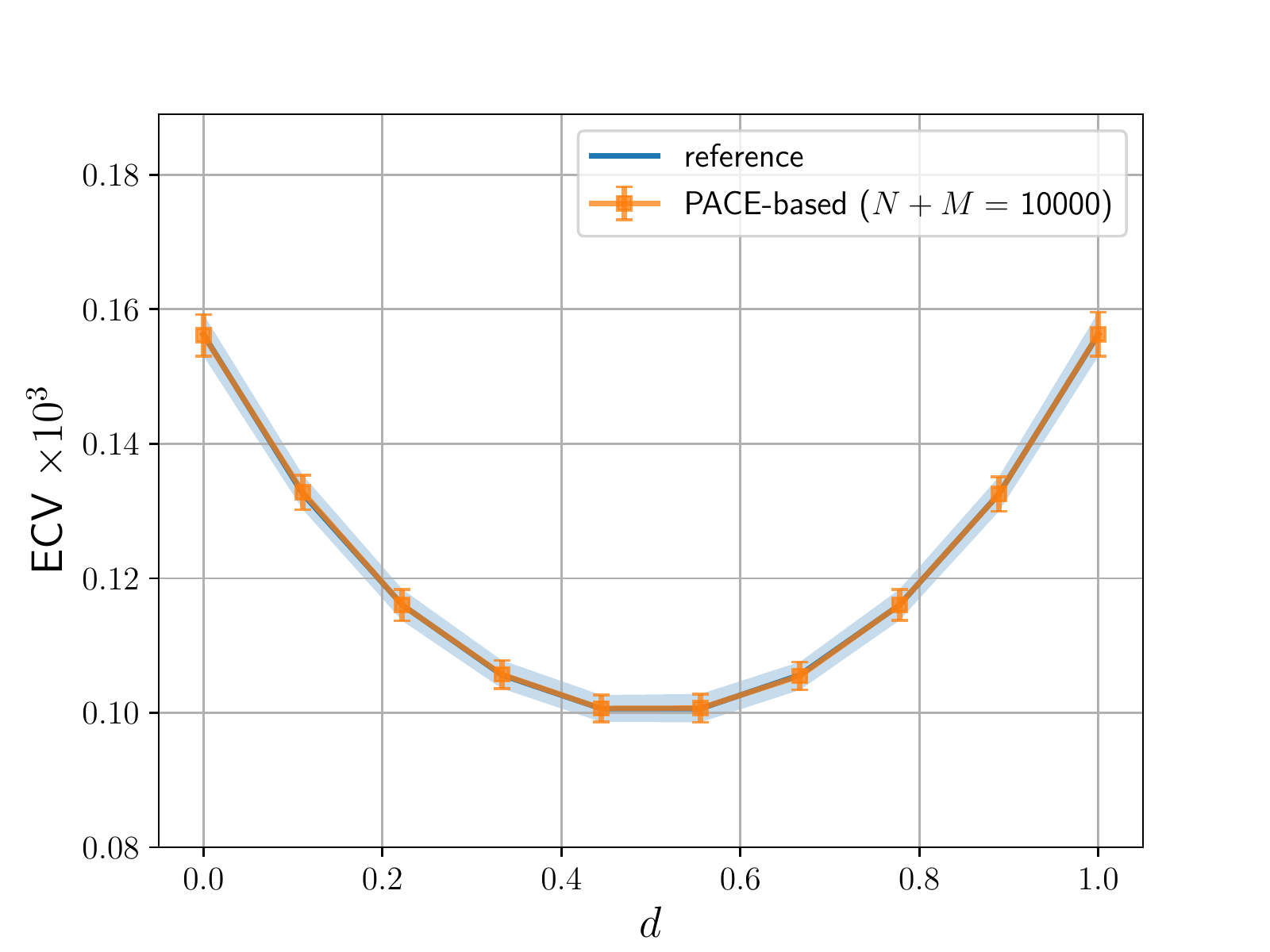}
		\caption{$~$}
	\end{subfigure}
	\begin{subfigure}[b]{0.49\textwidth}
		\includegraphics[scale=0.4]{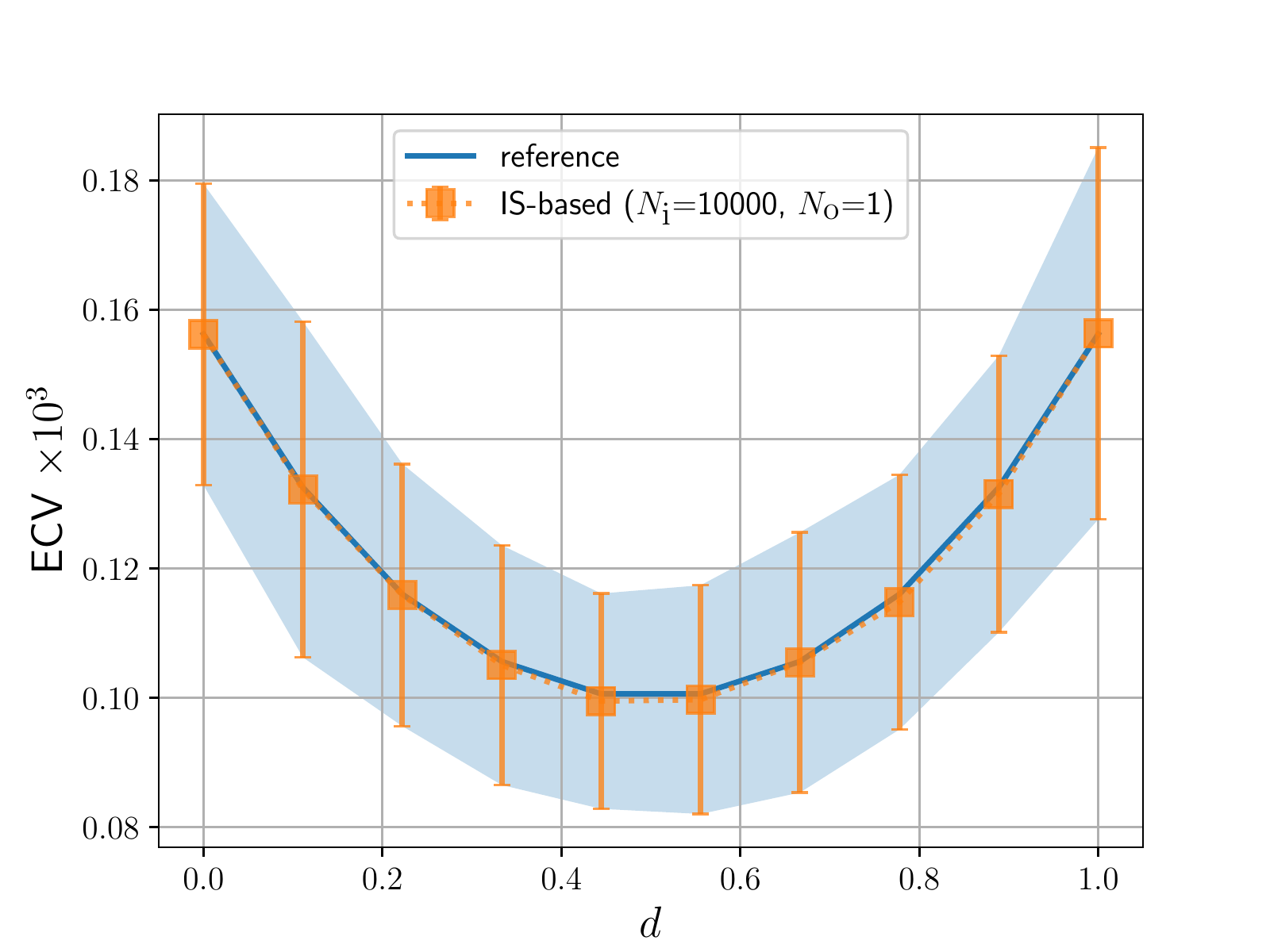}
		\caption{$~$}
	\end{subfigure}
	
	\caption{Comparison of the PACE- and IS-based approaches for
		estimating ECV with
		$\rvXi \sim \mathcal{N}(0, 0.01^2)$ using 10,000 samples:
		(a) PACE-based approach without data augmentation and (b) IS-based approach.
		The error bar plot depicts the mean and standard deviation of the empirical ECV values
		estimated from 1000 statistically independent MC simulations.}
	\label{fig:ce_1d_many_d}
\end{figure}

Fig.~\ref{fig:ce_1d_many_d} clearly illustrates
that the PACE-based approach  accurately
estimates the A-optimal DOE, \emph{i.e.}, $d_A= 0.5$, with 10,000 samples.
However, the IS-based approach results in significant errors in identifying the optimal design due to its statistical errors in estimating the tECV.

\subsection{High-dimensional case}\label{sec:nd}
In this section, we address the high-dimensional linear-Gaussian scenario.
The setting can be described as follows:
\begin{equation}
	\begin{aligned}
		h(\vq, \vd) &= \dfrac{1}{(d-0.5)^2+1} \; \rvQ,\quad d\in [0,1], \\
		\rvQ &\sim \mathcal{N}(\vek{0}_n,\mathbf{I}_n), \\
		\rvXi &\sim \mathcal{N}(\vek{0}_n, 0.1^2\mathbf{I}_n),
	\end{aligned}
\end{equation}
where $\vek{0}_n \in \sR^\dimq$ and
$\mathbf{I}_n \in \sR^{n\times n}$ represent a zero vector and an identity matrix, respectively.
We perform an analysis similar to the one presented in Section~\ref{sec:1d}.

\begin{figure}[!ht]
	\centering
	\begin{subfigure}[b]{0.49\textwidth}
		\includegraphics[scale=0.4]{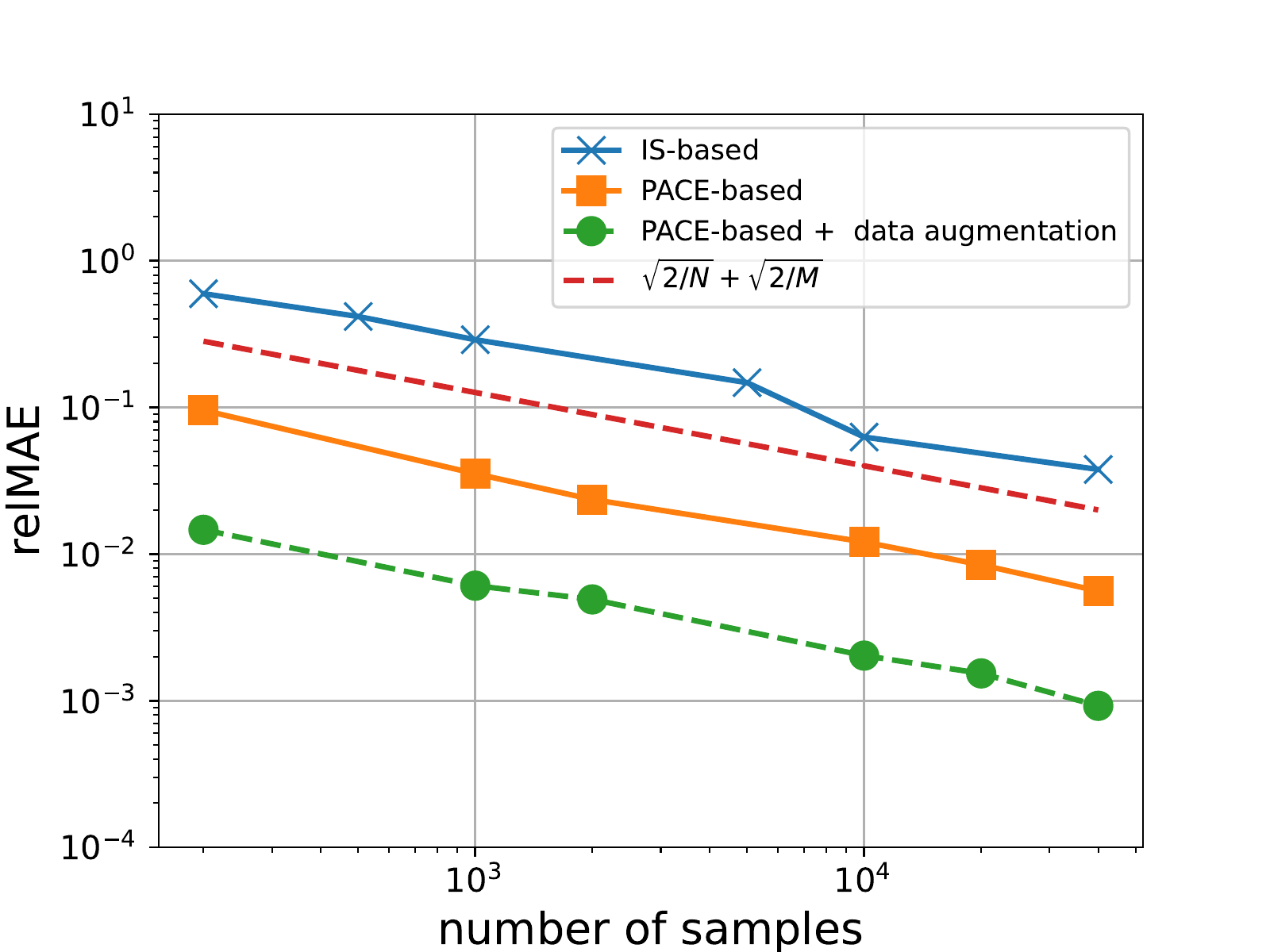}
		\caption{}
	\end{subfigure}
	\begin{subfigure}[b]{0.49\textwidth}
		\includegraphics[scale=0.4]{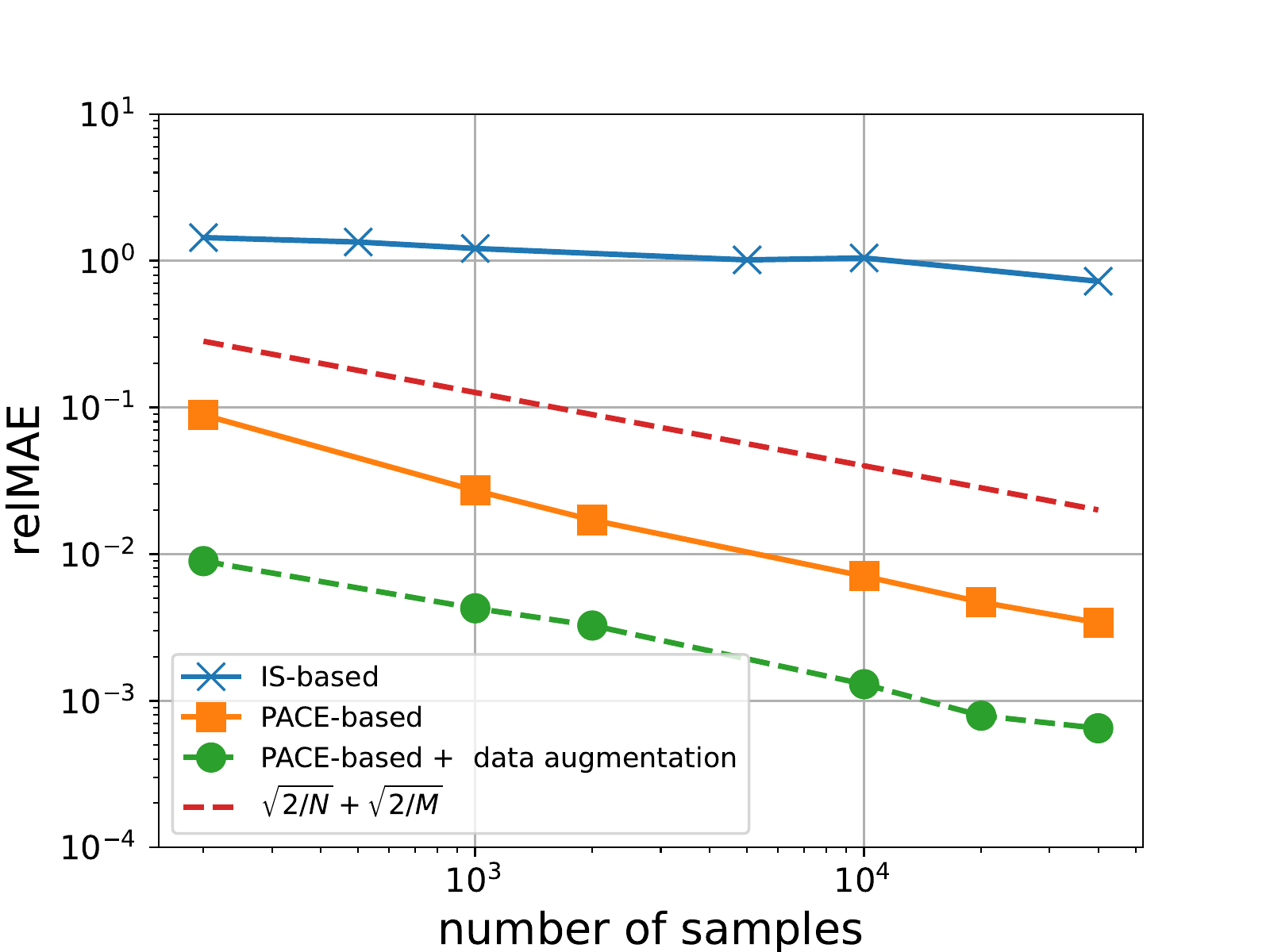}
		\caption{$~$}
	\end{subfigure}
	
	\centering
	\begin{subfigure}[b]{0.49\textwidth}
		\includegraphics[scale=0.4]{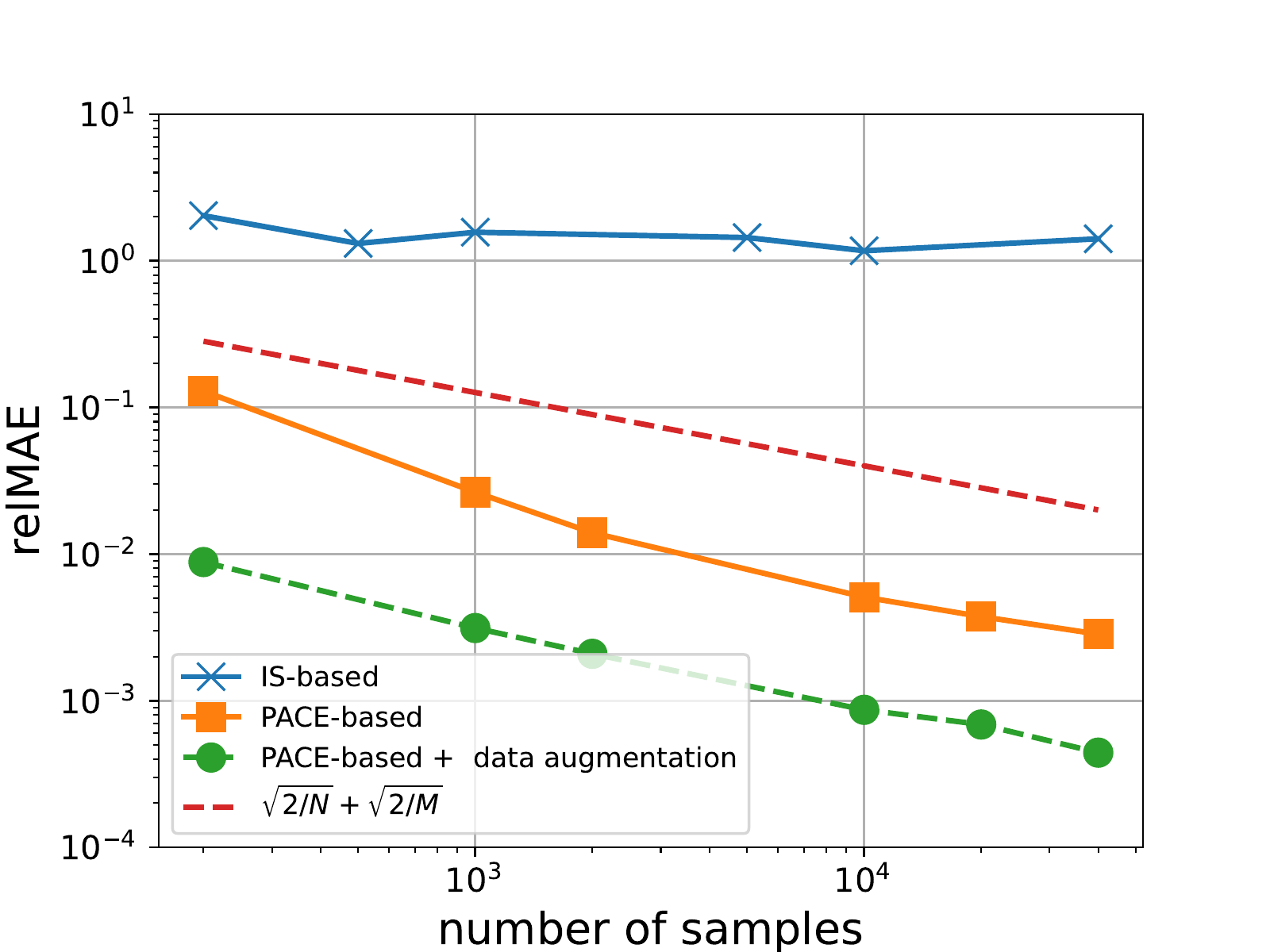}
		\caption{$~$}
	\end{subfigure}
	\begin{subfigure}[b]{0.49\textwidth}
		\includegraphics[scale=0.4]{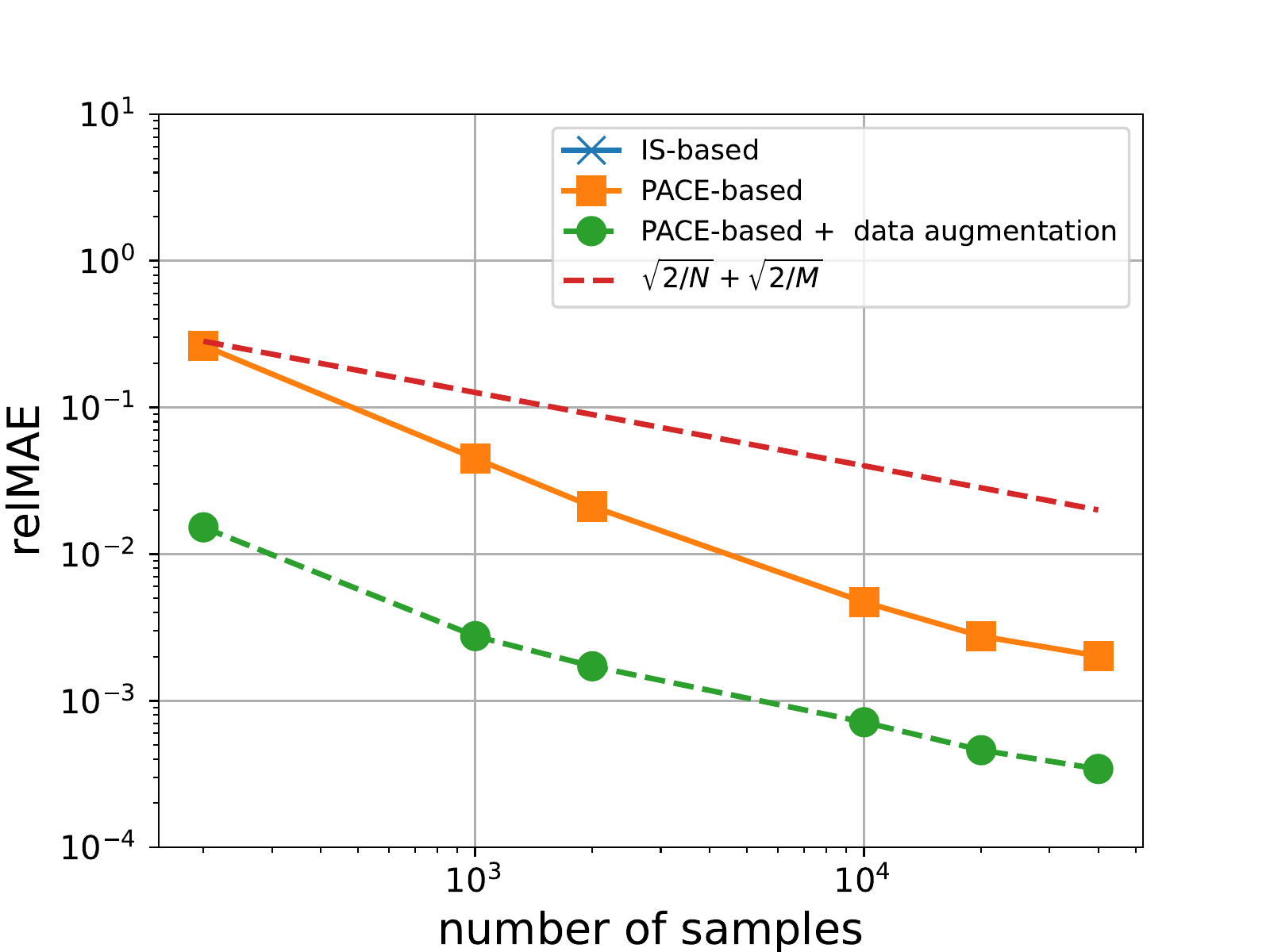}
		\caption{$~$}
	\end{subfigure}
	
	\caption{Comparison of the PACE- and IS-based approaches with high-dimensional linear examples:
		(a) $n=2$, (b) $n=5$, (c) $n=10$, and (d) $n=20$. For $n=20$, all 100 statistically independent MC simulations that uses the IS-based approach return a \emph{not-a-number} error. Consequently, the numerical result of the IS-based approach for $n=20$ is not available.}
	\label{fig:ce_nd}
\end{figure}

Fig.~\ref{fig:ce_nd} illustrates the relative MAE for different values of $n \in \{2, 5, 10, 20\}$. 
The numerical results demonstrate that the IS-based approach requires an extremely large number of samples for high-dimensional cases.
Specifically, for  $n \geq 5$, we do not observe the convergence of the IS-based approach.
For $n = 20$, the IS-based approach frequently returns a \emph{not-a-number} error due to the decreasing evidence term in the Bayes' theorem  (see Eq.~(\ref{eq:evidence})).
In contrast, our PACE-based approach effectively mitigates the curse of dimensionality.

\section{Stochastic optimization for the continuous design domain}\label{sec:continuous_domain}
The PACE approach can be used directly to solve the A-ODE problem in a discrete and finite design domain.
By estimating the tECV for each experimental design candidate using PACE, we identify the optimal solution as the design with the minimum tECV. 
In this section, we shift our focus to the continuous design domain.
We assume that the gradient of the measurement map \emph{w.r.t.} the design parameters, $\nabla_{\vd} h$, exists and can be computed numerically.
To solve the A-ODE problem efficiently, we apply stochastic gradient descent techniques on top of the PACE-based MC estimator of tECV.

When stochastic gradient optimization algorithms are used to solve the A-optimal DOE prob-lem, the computation of the gradient of the tECV \emph{w.r.t.} the design parameters, denoted as ${\nabla_{\vd} V_{\vd}}$, can pose a challenge. To address this issue, we propose a nonlocal approximation of the CE.
Our approach involves approximating the CE within a neighborhood surrounding a given design parameter vector to enable the evaluation of ${\nabla_{\vd} V_{\vd}}$.
We utilize ANNs as the approximation tool, although other regression techniques can also be incorporated into our method.

Let $f_{\text{N}}(\cdot; \weight): \sR^{\dimy+\dimd} \rightarrow \sR^\dimq$ be an ANN map,
where $\weight$ denotes the hyperparameters of the map.
The input of the proposed ANN is the concatenation of an $\dimy$-dimensional vector $\vy$ and a $\dimd$-dimensional vector $\vd$. By combining Eqs.~(\ref{eq:AoptimalED}), (\ref{eq:cm_orthogonal_projection}),  and (\ref{eq:theorem1}) the A-optimal design parameter vector, $\vd_{\text{A}}$,
can be obtained by solving the following optimization problem
\begin{equation}\label{eq:d-optimal-a} 
	\vd_{\text{A}} = \arg_{\vd} \; \min_{\vd, \weight} \;  \expectation{\left\lVert \rvQ -f_{\text{N}}(\rvY_{\vd}, \vd; \weight)  \right\rVert_2^2}
\end{equation}
which minimizes the MSE \emph{w.r.t.} the ANN's hyperparameters and design parameters.
Notably, the optimization problem that approximates the CE and the problem that seeks the A-optimal DOE minimize
the same objective function.

During the course of the iterative optimization process to solve Eq.~(\ref{eq:d-optimal-a}), the ANN may need to be retrained when the algorithm provides a new candidate vector for $\vd_{\text{A}}$.
Here, we employ the transfer learning technique, \emph{i.e.}, the weights trained for the previous design candidate are used as the starting point for the current training process. 
This simple transfer learning technique efficiently reduces the training time and the amount of data required.
The reminder of this section provides a more detailed explanation of the nonlocal approximation of the CE in Section~\ref{sec:condtional_mean_approximation} and discusses the algorithms that are used for implementing our method in Section~\ref{sec:algorithm}.

\subsection{Nonlocal approximation of the conditional expectation}\label{sec:condtional_mean_approximation}

Given a design candidate $\vd_{\kappa}$, we use a PDF  $w_{\kappa}$ defined over the design domain $\mathcal {D}$ such that $w_{\kappa} (\vd) : = \frac{w(\vd ,  \vd_{\kappa})}{c}$. Here
$w(\vd , \vd_{\kappa})$ is a weight function such as a kernel, and $c$ is a normalization constant defined as $c:= \int_{\mathcal{D}}{w(\vd , \vd_{\kappa})} \mathrm{d} \vd \, < \infty$.
In our notation, $\kappa \in \sN_{+}$ denotes the iteration counter of the main procedure, which  will be discussed in Section~\ref{sec:algorithm}.

To approximate the CE in the neighborhood of a design $\vd_{\kappa}$,
we determine the hyperparameters $\weight_\kappa$ of the ANN $f_{\text{N}}$  by solving the following optimization problem:
\begin{equation}\label{eq:theta_op}
	\weight_\kappa = \arg \quad \min_{\weight} \quad {L}_{\kappa}(\weight),
\end{equation}
where $L_{\kappa}$ is the weighted MSE function defined as
\begin{equation}\label{eq:weighted_mse}
	L_{\kappa}(\weight) = \int_{\mathcal{D}} \expectation{\left \lVert\rvQ -f_{\text{N}}(\rvY_{\vd}, \vd; \weight)
		\right \rVert_2^2 } w_{\kappa} (\vd) \dint \vd.
\end{equation}
We use the variance reduction technique, similar to Eq.~(\ref{eq:mc_vd}), to approximate the function $L_{\kappa}$ as
\begin{equation}\label{eq:hat_weighted_mse}
	\widehat{L}^{\text{vr}}_{\kappa}(\weight) = \dfrac{1}{N\times a}
	\sum_{i=1}^{N} \sum_{j=1}^{a}
	\left \rVert \vq^{(i)} - f_{\text{N}}(\vy^{(i,j)},d^{(i)}; \weight) \right\rVert_2^2,
\end{equation}
where $\{\vq^{(i)}\}_{i=1}^N$ are the \emph{i.i.d.} samples of RV $\rvQ$, $\{\vd^{(i)}\}_{i=1}^N$ are the \emph{i.i.d} samples generated according to the PDF $w_{\kappa}$, and $\vy^{(i,j)} = h(\vq^{(i)}, \vd^{(i)}) + \xi^{(i,j)}$, where  $\xi^{(i,j)}$ is the
 \emph{i.i.d.} sample of RV $\rvXi$.

A simple choice for the kernel function is the density function of a multivariate Gaussian distribution,  where the marginal variances are used as tuning parameters.
Intuitively, increasing the marginal variances of PDF $w_{\kappa}$ reduces the error of the CE approximation over the entire design domain; however, this  also requires a considerably large number of samples to accurately estimate the weighted MSE $\widehat{L}^{\text{vr}}_{\kappa}$.
In our algorithm, which is discussed in more detail in Section~\ref{sec:algorithm}, for the $\kappa$-\emph{th} iteration, we only need to evaluate the tECV and its derivative in the neighborhood of $\vd_{\kappa}$.
Therefore, a required standard deviation of PDF $\eta_{\kappa}$ is much smaller than the characteristic length of the design domain,  which improves the computational efficiency by reducing the number of samples required to estimate the weighted MSE $\widehat{L}^{\text{vr}}_{\kappa}$.

\subsection{Algorithm}\label{sec:algorithm}
In this subsection, we present a stochastic gradient descent algorithm to solve Eq.~(\ref{eq:d-optimal-a}). The algorithm utilizes the nonlocal approximation of the CE to estimate the tECV $V_{\vd}$ and its derivative  $\nabla_{\vd} V_{\vd}$.
A typical $\kappa$-\emph{th} iteration of the proposed algorithm consists of two steps:
\begin{itemize}
	\item [i)] Given a design parameter candidate $\vd_{\kappa-1}$,
	we apply \emph{transfer learning} and solve Eq.~(\ref{eq:theta_op}) to update the hyper-parameters $\weight_{\kappa}$,
	\item [ii)] We update the design parameter candidate using the gradient of the MSE
	\emph{w.r.t} $\vd$, \emph{i.e.}, $\nabla_{\vd} \,  \expectation{\lVert\rvQ -f_{\text{N}}(\rvY_{\vd}, \vd; \weight_{\kappa})\rVert_2^2 }$.
	This step returns an updated candidate $\vd_{\kappa}$.
\end{itemize}
Algorithm~\ref{algorithm:main} summarizes the main procedure, which calls
the suboptimization procedures that corres\-pond to steps i) and ii) described in Algorithms~\ref{algorithm:local_approximation_CE} and
\ref{algorithm:sgd_design}, respectively.
The suboptimization procedures in each step are executed using the Adam algorithm~\cite{kingma2014adam}.

\begin{algorithm}[H]
	\caption{Main}
	\label{algorithm:main}
	\begin{algorithmic}[1]
		\Require Observational model, its derivatives ($\nabla_{\vd} h$), and number of iteration $K$
		\State Initiate a value for $\vd_{0}$
		\For{$\kappa$ from 1 to $K$}
		\State approximating the conditional expectation using Algorithm~\ref{algorithm:local_approximation_CE} $\rightarrow$ $\weight_{\kappa}$
		\State solving $\arg \; \min_{\vd}  \expectation{\lVert\rvQ -f_{\text{N}}(\rvY_{\vd}, \vd; \weight_{\kappa})\rVert_2^2} $ using Algorithm~\ref{algorithm:sgd_design} $\rightarrow$ $\vd_\kappa$
		\EndFor\\
		\Return $\vd_{\text{A}} \leftarrow \vd_{K}$
	\end{algorithmic}
\end{algorithm}

In Algorithm~\ref{algorithm:local_approximation_CE}, we generate \emph{i.i.d.} samples $\{\vd^{(i)}\}_{i=1}^N$, $\{\vq^{(i)}\}_{i=1}^N$, and $\{\xi^{(i)}\}_{i=1}^N$ following PDFs $w_{\kappa}$, $\pi_{\rvQ}$, and $\pi_{\rvXi}$, respectively.
Next, we use those samples to compute the samples of the observational RV, $\vek{y}^{(i, j)} = h(\vq^{(i)},\vd^{(i)}) + \xi^{(i, j)}$ for $j =1,\dots, a$.
Finally, we employ the Adam algorithm to train the ANN. We use the dataset $\left \{ \left(\vq^{(i)}, \vd^{(i)}, \vek{y}^{(i,j)} \right ) \right \}_{i =1, \dots, N; \; j =1, \dots a}$ and the loss function $\widehat{L}^{\text{vr}}_{\kappa}$ (defined in Eq.~(\ref{eq:hat_weighted_mse})) for this training.

To reduce the numbers of epochs ($e_1$) and data samples ($N\times a$), we employ the \emph{transfer learning} technique \cite{torrey2010, weiss2016}, which reuses the hyperparameters obtained from the previous iteration, $\weight_{\kappa-1}$, as the initial values for the current iteration $\kappa$.

\begin{algorithm}[H]
	\caption{Nonlocal approximation of the CE operator}
	\label{algorithm:local_approximation_CE}
	\begin{algorithmic}[1]
		\Require $\vd_\kappa$, weight function $w(\cdot ; \vd_\kappa)$, number of epoch $e_1$, and learning rate $\alpha_1$
		\State Generate $\{\vd^{(i)}\}_\itn \sim w_{\kappa} (\vd)$, $\{\vq^{(i)}\}_\itn \sim \pi_{\rvQ}$, $\{\xi^{(i,j)}\}_{i =1, \dots, N; \; j =1, \dots a} \sim \pi_{\rvXi}$
		\State Evaluate $\{h(\vq^{(i)}, \vd^{(i)})\}_{i=1}^N$ 
		\State Compute samples $\{\vy^{(i,j)}\}_{i =1, \dots, N; \; j =1, \dots a}$ where
		$\vek{y}^{(i,j)} = h(\vq^{(i)},\vd^{(i)}) + \xi^{(i,j)}$
		\State Collect data into a dataset  
		$D^{\text{vr}}_N= \left \{ \left (\vq^{(i)}, \vd^{(i)}, \vek{y}^{(i,j)} \right) \right\}_{i =1, \dots, N; \; j =1, \dots a}$
		\State Train the ANN using the initial weights $\weight_{\kappa - 1}$ with $e_1$ epoch, Adam algorithm, learning rate $\alpha_1$, loss function $\widehat{L}^{\text{vr}}_{\kappa}$
		defined in Eq.~(\ref{eq:hat_weighted_mse}) and dataset $D^{\text{vr}}_N$ \\
		\Return Hyper-parameters $\weight_{\kappa}$ of ANN $f_{\text{N}}$
	\end{algorithmic}
\end{algorithm}

In Algorithm \ref{algorithm:sgd_design}, we generate \emph{i.i.d.} samples $\{\vq^{(i)}\}_{i=1}^M$ following PDF $\pi_{\rvQ}$.
The dataset $\{\vq^{(i)}\}_{i=1}^M$ is divided into different batches of size $M_b$, denoted as
$\{B_1,\dots, B_{\beta}\}$.
For each batch $B$ and design candidate $\vd_{0}$, we compute the corresponding samples
of the observational RV as $\vek{y}^{(i,j)} = h(\vq^{(i)},\vd_{0}) + \xi^{(i,j)}$,
where $\vq^{(i)}\in B$, $j =1,\dots, a$, and $\xi^{(i,j)}$ are the \emph{i.i.d.} samples of the RV $\rvXi$.
We then update the design parameter using
\begin{equation}\label{eq:update_design_candidate}
	\vd' = \vd_{0} - \dfrac{\alpha_2}{M_b\times a}
	\sum_{i=1}^{M_b} \sum_{j=1}^{a}\nabla_{\vd}
	\Bigl[\lVert q - f_{\text{N}}(h(\vq,d) +\xi,\vd) \rVert_2^2\Bigr](q^{(i)}, \vy^{(i,j)},\vd_{0}),
\end{equation}
where $\alpha_2$ is a learning rate. Here, $\nabla_{\vd}  \Bigl[\lVert q - f_{\text{N}}(h(\vq,d) +\xi,\vd) \rVert_2^2\Bigr](q^{(i)}, \vy^{(i,j)},\vd_{0})$
represents the gradient of $\lVert q - f_{\text{N}}(h(\vq,d) +\xi,\vd) \rVert_2^2$ \emph{w.r.t.} $\vd$ and is 
evaluated at $(q^{(i)}, \vy^{(i,j)},\vd_{0})$. The value of $\vd'$ is then assigned to $\vd_{0}$ for the next batch.
\begin{algorithm}[H]
	\caption{Minimizing the MSE $ \expectation{\lVert\rvQ -f_{\text{N}}(\rvY_{\vd}, \vd; \weight_{\kappa}) \rVert_2^2}$
		\emph{w.r.t.} $\vd$}
	\label{algorithm:sgd_design}
	\begin{algorithmic}[1]
		\Require Observational model ($h$) and its derivatives ($\nabla_{\vd} h$),
		map $f_{\text{N}}(\cdot; \weight_{\kappa})$ from Algorithm~\ref{algorithm:local_approximation_CE},
		$\vd_{\kappa}$, number of samples $M$, number of epoch $e_2$, batch-size $M_b$, and learning rate $\alpha_2$
		\State $\vd_{0} \leftarrow \vd_{\kappa}$
		\State Generate samples $\{\vq^{(i)}\}_{i=1}^M$ following prior PDF $\pi_{\rvQ}$
		\For{$e=1$ to $e_2$}
		\For{$b=1$ to $M/M_b$}
		\State Generate samples $\{\xi^{(i,j)}\}_{i=1,\dots, M_b;\; j =1,\dots, a}$ of the measurement error $\rvXi$
		\State Evaluate $\vy^{(i,j)} = h(\vq^{(i)}, d_{0}) + \xi^{(i,j)}$
		\State Updating the design candidate using Adam algorithm based on  Eq.~(\ref{eq:update_design_candidate})
		\State $\vd_{0} \leftarrow \vd'$
		\EndFor
		\EndFor\\
		\Return $\vd_{\kappa+1} \leftarrow \vd_{0}$
	\end{algorithmic}
\end{algorithm}

In Appendix~\ref{appendix:gradients}, we use the chain rule to derive the formulation of the gradient 
$\nabla_{\vd} \Bigl[\lVert \vq - f_{\text{N}}(h(\vq)+\xi,\vd) \rVert_2^2\Bigr](q^{(i)}, \vy^{(i,j)} ,\vd_{0})$ in terms of
the gradients of the observational map $h$ and the ANN $f_{\text{N}}$ \emph{w.r.t} $\vd$.
Notably, the gradients of the ANN function $f_{\text{N}}$ are obtained using the backpropagation \cite{GrieWal08}.

\section{Numerical experiment: Electrical impedance tomography}
\label{sec:eit}
Electrical impedance tomography (EIT) aims to determine the conductivity of a closed body by measuring the potential of electrodes placed on its surface in response to applied electric currents.
In this section, we revisit an experimental design problem previously studied in \cite{Beck2018, Beck2020, CARLON2020}. The objective of the experiment is to recover the fiber orientation in each orthotropic ply of a composite laminate material. This is achieved by measuring the potential when applying  low-frequency electric currents through electrodes attached to the body surface. 
The formulation of the EIT experiment is described in Section~\ref{sec:FE}, where we adopt the entire electrode model from~\cite{somersalo1992existence}.

Previous studies employed the maximization of EIG as the optimality criterion and estimated the EIG using the IS technique combined with the Laplace approximation of the posteriors.
In Section~\ref{sec:eit_optimization}, we analyze our PACE-based approach for estimating the tECV and compare the obtained results with those using the IS-based technique.
Subsequently, we utilize the optimization algorithms discussed  in Section~\ref{sec:continuous_domain} to solve the DOE problem and present the numerical results.

\subsection{Formulation and finite element model of EIT experiment}\label{sec:FE}
We consider a rectangular body ${B}$ which has boundary $\partial B$ and
consists of two plies ${B}=\bigcup_{i=1}^2 {B}_i$.
The exterior boundary surfaces  are equipped with $N_{\text{el}}=10$ electrodes on the area $E_1, \dots, E_{N_{\text{el}}}$.

These electrodes are used to apply known currents,
and they also allow us to measure the electric potential.
The configuration used in this study is illustrated in Fig.~\ref{fig:eit}.
\begin{figure}[!ht]
	\centering
	\includegraphics[scale=0.2]{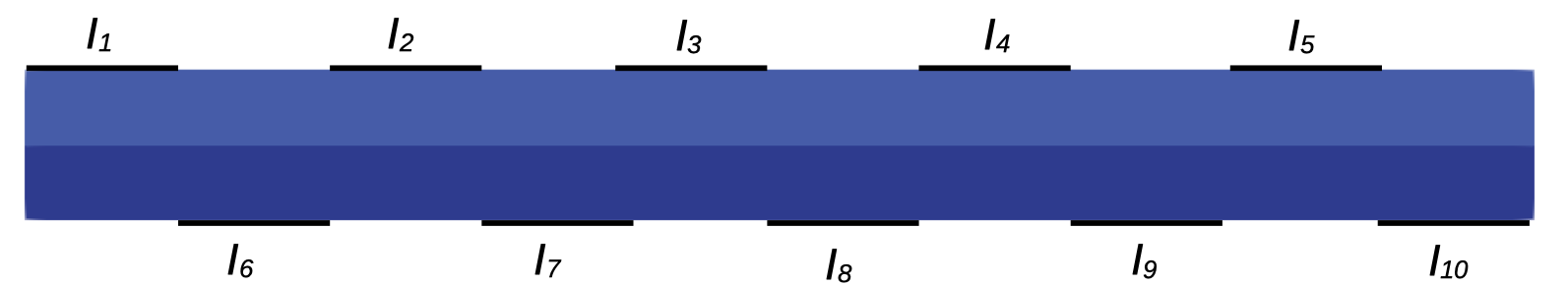}
	\caption{Illustration of the EIT experiment for a two-ply composite sample. The black rectangles indicate the electrode positions, and $I_1, \dots, I_{10}$ are their applied currents.}
	\label{fig:eit}
\end{figure}

The fiber orientation over body ${B}$ is modeled as a random field
$\eta: {B} \times \varOmega \rightarrow [-\pi, \pi]$ as follows:
\begin{equation}
	\eta(x, \omega) =
	\begin{cases}
		& \eta_1(\omega)\quad \text{in} \; {B}_1, \\
		& \eta_2(\omega) \quad \text{in}\; {B}_2.
	\end{cases}
\end{equation}
The RV $\rvQ$ in this example is a random vector valued in  $[-\pi, \pi]^2$ that represents the prior distribution of the fiber orientations of the plies, defined as
\begin{equation}\label{eq:eit_rvq}
	\rvQ (\omega) := [\eta_1(\omega), \eta_2(\omega)]^\top.
\end{equation}

We represent the quasistatic current flux and the quasi-static potential as random fields as
$\zeta: {B} \times \varOmega \rightarrow \sR^3$ and $u: {B} \times \varOmega \rightarrow \sR$,
respectively. The electrical current field
and the  potential field satisfy the following PDE:
\begin{equation}\label{eq:eit_pde}
	\begin{aligned}
		\nabla \cdot \zeta &= 0,\\
		\zeta &= \sigma (\eta)\; \nabla u,
	\end{aligned}\;
\end{equation}
where $\sigma$ is the conductivity that depends on the fiber orientation of the plies.
The relation between the conductivity and the fiber orientations $\eta$ is given as
\begin{equation}
	\sigma (\eta) = \mathcal{R}(\eta)^\top \cdot \bar{\sigma} \cdot  \mathcal{R}(\eta),
\end{equation}
where $\mathcal{R}(\eta)$ is the rotation matrix,
\begin{equation}
	\mathcal{R}(\eta) =
	\begin{bmatrix}
		\cos(\eta) & 0 & \sin(\eta) \\
		0 & 1 & 0 \\
		-\sin(\eta) & 0 & \cos(\eta)
	\end{bmatrix}.
\end{equation}
Here, $\bar{\sigma}$ is a $3\times 3$ constant matrix,
which we set to $\bar{\sigma} = \operatorname{diag}(10^{-2}, 10^{-3}, 10^{-3})$.
The boundary conditions are given by
\begin{equation}
	\begin{aligned}\label{eq:pde.bc.1}
		\zeta\cdot n=0 &, \quad \text{on}\, \partial {B} \setminus (\cup E_l), \\
		\int_{E_l}\zeta\cdot n\dint{}x=I_l &, \quad l =  1,\dots, N_{\text{el}}, \\
		\frac{1}{E_l}\int_{E_l}u\dint{}x+z_l\int_{E_l}\zeta \cdot{n}\dint{}x=
		U_l &, \quad   l =  1,\dots, N_{\text{el}},
	\end{aligned}
\end{equation}
where $n$ is the unit outward normal, and $I_l$ and $U_l$ are the applied current and observable electric potential at electrode $E_l$, respectively.
The first boundary condition states the no-flux condition, and the second one implies that the total injected current through each electrode is known.
The third boundary condition captures the surface impedance effect, meaning that the shared interface between the electrode and material has an infinitesimally thin layer with a surface impedance of $z_l$.
In this study, we set to $z_l=0.1$ for the surface impedance.
We use the following two constraints (Kirchhoff law of charge conservation and ground potential condition) to guarantee the well-posedness:
\begin{equation}\label{eq:pde.bc.2}
	\sum_{l=1}^{N_{\text{el}}}I_l=0 \quad \text{and} \quad \sum_{l=1}^{N_{\text{el}}}U_l=0.
\end{equation}

To solve the PDE stated in Eq.~(\ref{eq:eit_pde}), we develop a 2D FE model using a quadratic mesh of size $50\times 6$ elements.
Details on the weak formulation of the FE model  are described in  Appendix~\ref{appendix:fem}.

\subsection{Numerical results}\label{sec:eit_optimization}

\paragraph{Setting}
Ten electrodes are placed on the surface of the rectangular domain ${B}=[0, 20]\times [0, 2]$, five on the top and five on the bottom of the surface.
The fiber orientations $\eta_1$ in ply ${B}_1$ and $\eta_2$ in ply ${B}_2$ are the quantities of interest with the following assumed prior distributions:
\begin{equation}
	\eta_1\sim \mathcal{U}\left(\frac{\pi}{4.5}, \frac{\pi}{3.5}\right), \quad \eta_2 \sim \mathcal{U}\left(-\frac{\pi}{3.5}, -\frac{\pi}{4.5}\right).
\end{equation}
The observational model is given as
\begin{equation}
	\begin{aligned}
		Y_{\vd} & = h(\rvQ, \vd) + \rvXi,\\
		\text{where} \quad  h(\rvQ, \vd) & := [U_1(\rvQ, \vd), \dots,
		U_{10}(\rvQ, \vd)]^\top,
	\end{aligned}
\end{equation}
and $\rvQ$ is defined in Eq.~(\ref{eq:eit_rvq}).
We consider two cases of the observational error distribution, which are
$\rvXi\sim \mathcal{N}(0, 10^2 \mathbf{I}_{10})$ and $\rvXi\sim\mathcal{N}(0, 3^2 \mathbf{I}_{10})$,
\emph{i.e.}, the error standard deviations are about 5\% and 1.5\%
of the ground-truth value, respectively.
We treat the applied electrode currents as the design parameters
denoted as $\vd := [I_1, \dots, I_{9}]^\top$, where
we apply the condition $I_{10} = -\sum_{l=1}^{9} I_l$
in accordance with the Kirchhoff law of charge conservation.
We seek the A-optimal DOE parameterized by vector $\vd \in [-1, 1]^{9}$.

To accelerate our analysis, we construct a surrogate model using an ANN to approximate the observational model $h$. 
The input of this surrogate model consists of the fiber orientations $ [\eta_1, \eta_2]^\top$ and the applied electrode currents $
[I_1, \dots, I_{9}]^\top$, and its output is the vector of the potential at ten electrodes $ [U_1, \dots, U_{10}]^\top $. 
 The surrogate model is trained on a dataset obtained by running the FE model for a large number of input samples such that 
 $[\eta_1, \eta_2]^\top \sim \mathcal{U}\left([\frac{\pi}{4.5}, \frac{\pi}{3.5}] \times [-\frac{\pi}{3.5}, -\frac{\pi}{4.5}]\right)$ and 
 $[I_1, \dots, I_{9}]^\top \sim \mathcal{U}([-1, 1]^{9})$ conditioned by  $\lvert \sum_{l=1}^{9} I_l \rvert \leq 1$.

\subsubsection{Analysis of error in estimating the tECV}\label{sec:error_analysis}
Here we study the errors in estimating the tECV for a specified design parameter vector.
For the PACE-based approach, we implement three models: i) using the linear approximation for the CE (see Appendix~\ref{appendix:linear_approximation}), ii) using the ANN for approximating the CE without data augmentation, and iii) using the ANN for approximating the CE and applying the data augmentation method described in Section~\ref{sec:reduced_variance_estimator}.
The ANNs in cases ii and iii have two hidden layers of 100 neurons each.
The input and output layers have 10 and 2 neurons, respectively, which corresponds to the dimensions of vectors $\vy$ and $\vq$, respectively.
Dataset $D_N$ is split into two datasets for training and testing with a size ratio of 1:1.
We use the Adam algorithm \cite{kingma2014adam} with a learning rate of $0.0005$ and a batch size of $100$. 
The algorithm is set to run for a maximum of $10,000$ epochs and features an early-stop mechanism that is activated when the loss function shows no further improvement.
To compute the tECV, we fix $M=N$ based on the Proposal~\ref{proposition:error_estimation}.

As a closed form of the tECV is not available in this example, we employ the IS approach with $200,000$ \emph{i.i.d} samples, and we use the obtained result as the reference, \emph{i.e.}, $V$ in Eq.~(\ref{eq:numMAE}).
Fig.~\ref{fig:eit_approximation_error} depicts the relative MAEs in estimating the tECV for the optimal design found in \cite{Beck2020}. To obtain these estimated relative MAEs, we perform 50 statistically independent simulations to empirically compute the expectation operator in Eq.~(\ref{eq:numMAE}).

\begin{figure}[!ht]
	\centering
	\begin{subfigure}[b]{0.45\textwidth}
		\includegraphics[scale=0.4]{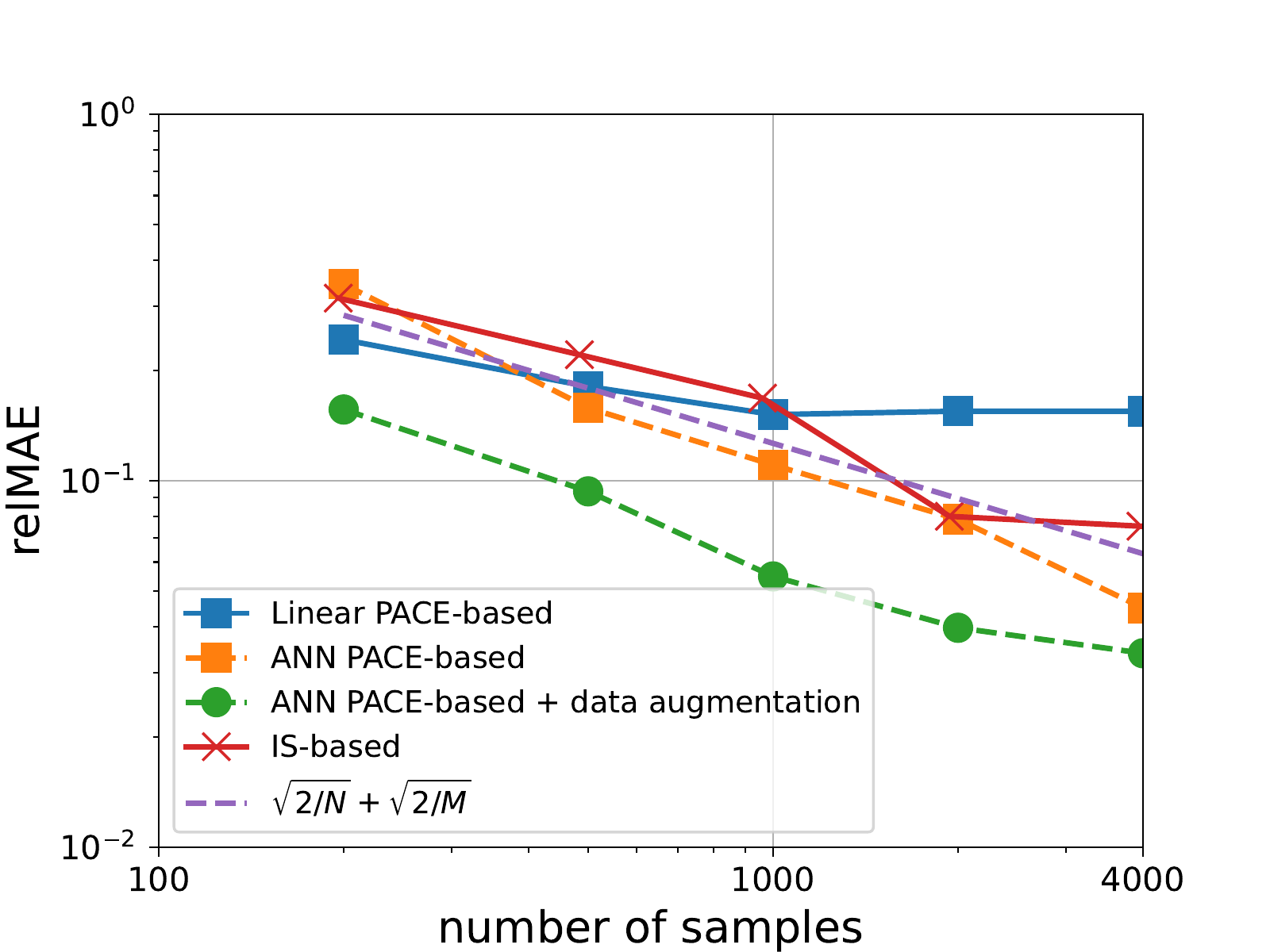}
		\caption{$~$}
	\end{subfigure}
	\begin{subfigure}[b]{0.45\textwidth}
		\includegraphics[scale=0.4]{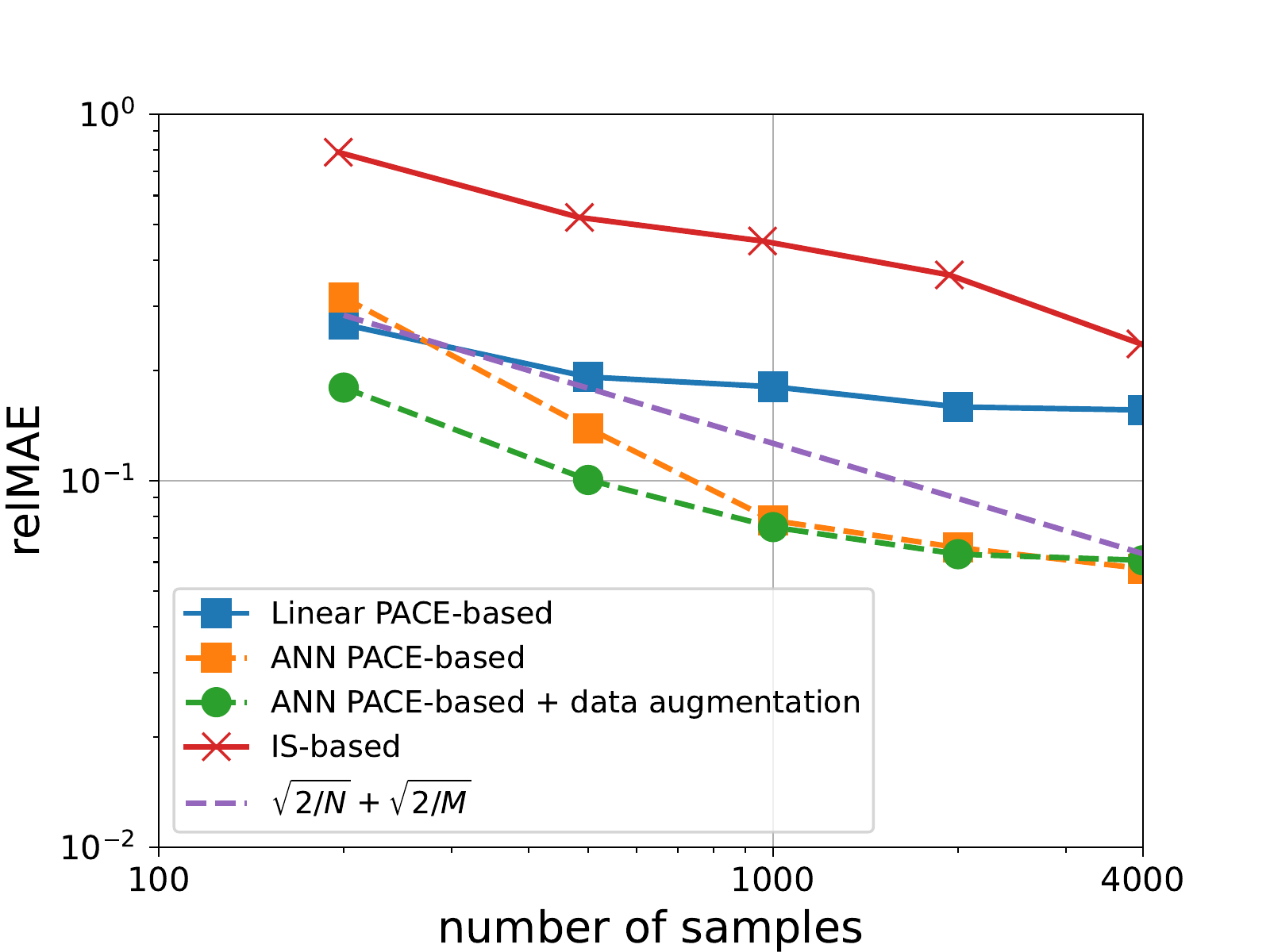}
		\caption{$~$}
	\end{subfigure}
	\caption{Relative MAE for estimating the tECV for $d = d_{\text{optimal}}$
		(a) $\rvXi\sim \mathcal{N} (0, 10^2 \mathbf{I}_{10})$, and (b) $\rvXi\sim \mathcal{N} (0, 3^2 \mathbf{I}_{10})$.
		For the PACE-based approach, the number of samples is equal to $N+M$.
		For the IS-based approach, the number of samples is equal to
		$(N_{\text{i}} + 1) \times N_{\text{o}}$.
		For the sake of simplicity, we fix $N=M$ and $N_{\text{i}} = N_{\text{o}}$. }
	\label{fig:eit_approximation_error}
\end{figure}

Because the measurement map is nonlinear, the CE linear approximation contains bias error, which  results in a  relative MAE of 20\%.
In contrast, we observe negligible bias error when we use the ANN for approximating the CE.
When comparing models ii and iii, we observe that the data augmentation method described in Section~\ref{sec:reduced_variance_estimator} plays a crucial role in reducing the estimation error, particularly for a small number of samples, as shown in Fig.\ref{fig:eit_approximation_error}.
For the case $N+M =200$, applying the data augmentation reduces the relative MAE by more than three-fold for $\rvXi\sim \mathcal{N} (0, 10^2 \mathbf{I}_{10})$.

With the IS-based approach, estimating the tECV requires a double-loop MC simulation, where the numbers of samples for the outer and inner loops are $N_{\text{o}}$ and $N_{\text{i}}$, respectively (see Appendix~\ref{appendix:ll-approach}). 
Here, we set $N_{\text{o}} = N_{\text{i}}$ for simplicity. Optimizing the ratio $N_{\text{o}} : N_{\text{i}}$ is problem dependent and beyond the scope of this study. 
Compared with the IS-based approach, the PACE-based approach  combined with data augmentation significantly decreases the number of required samples.
For example in Fig.~\ref{fig:eit_approximation_error}(a), to reach a relative MAE of 0.9\%,
the IS approach requires more than 4000 samples, while
the PACE-based approach using data augmentation needs only 1000 samples.
Notably, for the case $\rvXi\sim \mathcal{N} (0, 3^2 \mathbf{I}_{10})$ depicted in Fig.~\ref{fig:eit_approximation_error}(b),
although the computational efficiency of the PACE approach does not exhibit a significant change compared with the case in Fig.~\ref{fig:eit_approximation_error}(a), the performance of the IS approach deteriorates substantially. 
In this case, the statistical error of the  IS approach is even greater than that of the PACE approach using the linear approximation.
Moreover, the relative MAE of the PACE approach verifies its theoretical estimation, $\mathcal{O}\bigl(\sqrt{\frac{2}{\pi N}} +\sqrt{\frac{2}{\pi M}} \bigr)$ stated in Proposition~\ref{proposition:error_estimation}.

\subsubsection{Minimization of the tECV}
In this section we apply the algorithms developed in Section~\ref{sec:continuous_domain} to determine the A-optimal DOE. We approximate the CE nonlocally using an ANN having two hidden layers.
The input layers have 19 neurons that correspond to the sum of the dimensions of vectors $\vd\in [-1, 1]^9$ and $\vy\in \sR^{10}$.
The number of neurons in the output layer and each hidden layer are 2 and 100, respectively.
We choose the density of the normal distribution $\mathcal{N} (\text{0}_{9}, 0.2\times \mathbf{I}_{9})$ as
the kernel function used to evaluate the weighted loss function $L_{\kappa}$ in Eq.~(\ref{eq:weighted_mse}).
From this point onward, we keep $\rvXi\sim \mathcal{N} (0, 10^2 \mathbf{I}_{10})$ as was done \cite{CARLON2020}.

The main procedure described in Algorithm~\ref{algorithm:main} is implemented with 20 iterations.
In each iteration, we first use Algorithm~\ref{algorithm:local_approximation_CE} to approximate the CE over 1000 epochs.
We use $N=500$ samples to estimate the weighted MSE $L_{\kappa}$ (see Eq.~(\ref{eq:hat_weighted_mse})), and we  augment the dataset by 30 times using the data augmentation method described in Section~\ref{sec:reduced_variance_estimator}.
Other settings for training the ANN are kept identical to those in Section~\ref{sec:error_analysis}.
Given $\weight$, we execute Algorithm~\ref{algorithm:sgd_design} over 20 epochs using $N=25$ and a single batch.
The dataset is also augmented by 30-fold.
In Algorithm~\ref{algorithm:sgd_design}, the learning rate is gradually reduced from $0.1$ to $0.02$ to reduce the stochastic effect in the final epochs and ensure the convergence of the algorithm.
Algorithm~\ref{algorithm:sgd_design} returns $\vd_{\kappa+1}$, which is then utilized in the subsequent iteration.

For each iteration of the main procedure, Algorithms~\ref{algorithm:local_approximation_CE} and~\ref{algorithm:sgd_design} must evaluate the observational map 500 times. 
Additionally, the latter evaluates the gradient $\nabla_{\vd} h$ 500 more times. 
In total, Algorithm~\ref{algorithm:main} executes $20,000$ evaluations of the observational map and $10,000$ evaluations of the gradient $\nabla_{\vd} h$. 
The total number of evaluations of the observational map and its derivative is of the same order of magnitude as the estimations of the tECV for a single design (see Fig.~\ref{fig:eit_approximation_error}).
Our method gains significant computational efficiency owing to the nonlocal approximation of the CE and the applied transfer learning technique.
The performance of the algorithm in terms of the design vector and the tECV is depicted in Fig.~\ref{figure:opt_convergence}. Convergence can be clearly observed after 15 iterations. 

\begin{figure}[!ht]
	\centering
	\begin{subfigure}[b]{0.49\textwidth}
		\includegraphics[scale=0.44]{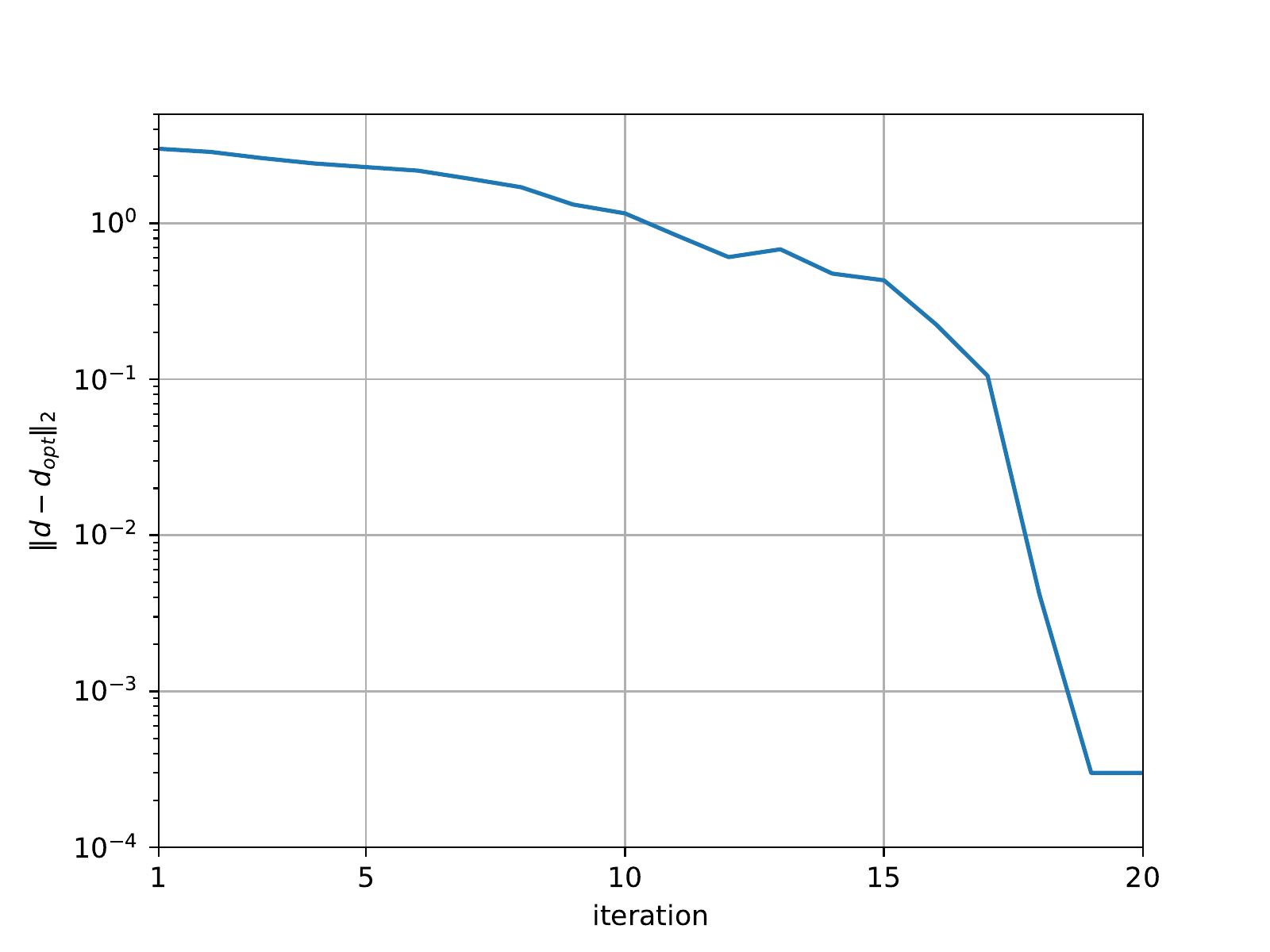}
		\caption{$~$}
	\end{subfigure}
	\begin{subfigure}[b]{0.49\textwidth}
		\includegraphics[scale=0.44]{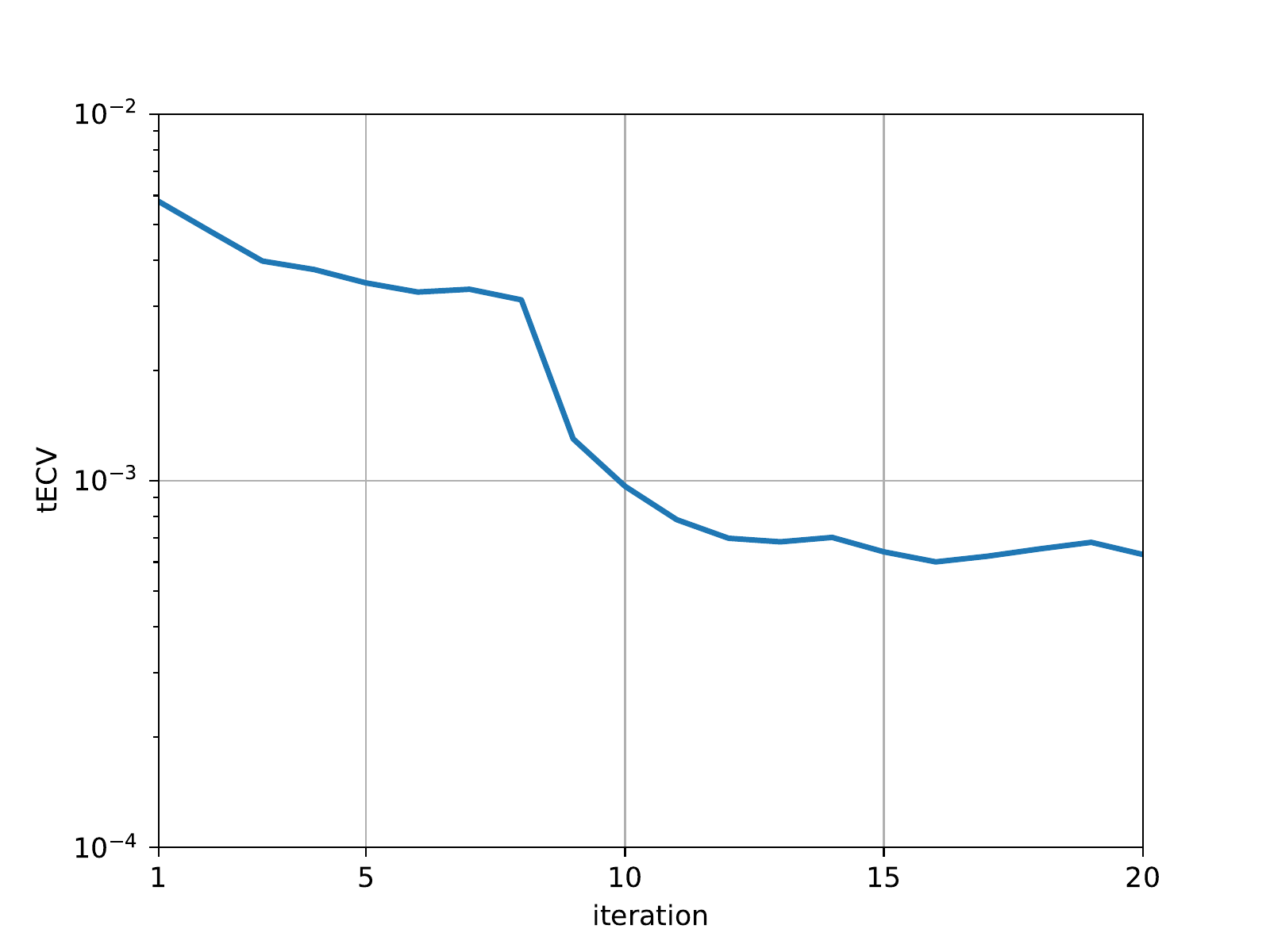}
		\caption{$~$}
	\end{subfigure}
	\caption{Convergence in terms of  (a) the design parameter and (b) the tECV. }
	\label{figure:opt_convergence}
\end{figure}

\begin{figure}[!ht]
	\centering
	\includegraphics[scale=0.2]{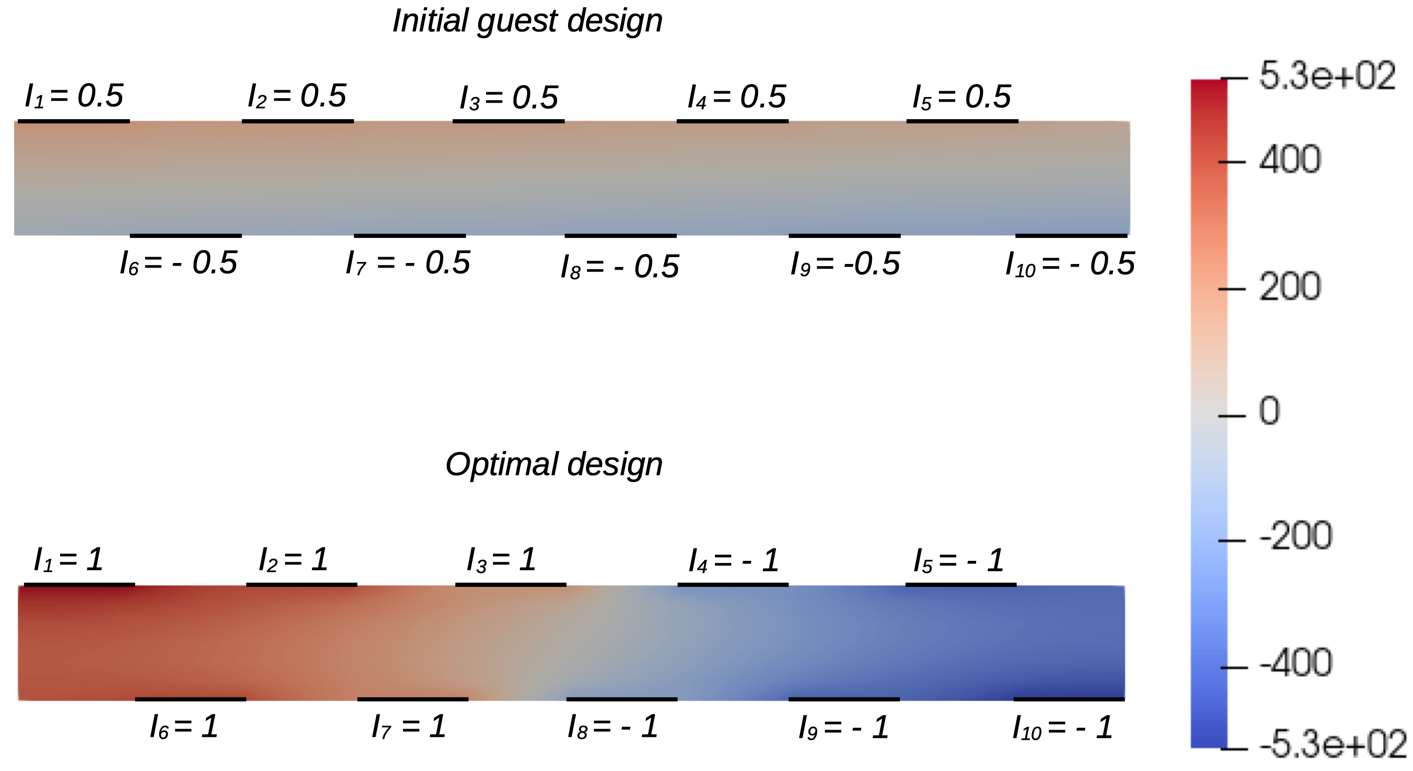}
	\caption{Potential fields
		solved using the FE model for the initial DOE (above) and the A-optimal DOE (below).}
	\label{fig:potential_field}
\end{figure}

\paragraph{A-optimal DOE}
The A-optimal design found by the algorithm closely resembles the solution  $\vd_{\text{A}} = [1,\; 1,\; 1,\; -1, -1,\; 1,\; 1, -1, -1]^{\top}$.
The potential fields solved using the FE model are illustrated in Fig.~\ref{fig:potential_field} for both the initial DOE and the A-optimal DOE.
The magnitude of the potential field under the A-optimal DOE, particularly at the electrodes, is considerably larger than that of the initial DOE. Consequently, the effect of the measurement error on the posterior variance is minimized.
In this example, minimizing the tECV leads to the DOE that is identical to the DOE maximizing the EIG, reported in \cite{CARLON2020}. This observation suggests that the conditional variance and the information gain are strongly correlated.

The typical posterior densities obtained from the initial DOE and the A-optimal DOE are depicted in Fig.~\ref{fig:eit_posteriors}(a) and (b), respectively.
The posterior mean that is inferred using the A-optimal DOE predicts the ground-truth value with a negligible error, whereas the posterior variance is significantly reduced compared with that of the initial DOE.

\begin{figure}[!ht]
	\centering
	\begin{subfigure}[b]{0.49\textwidth}
		\centering
		\includegraphics[scale=0.35]{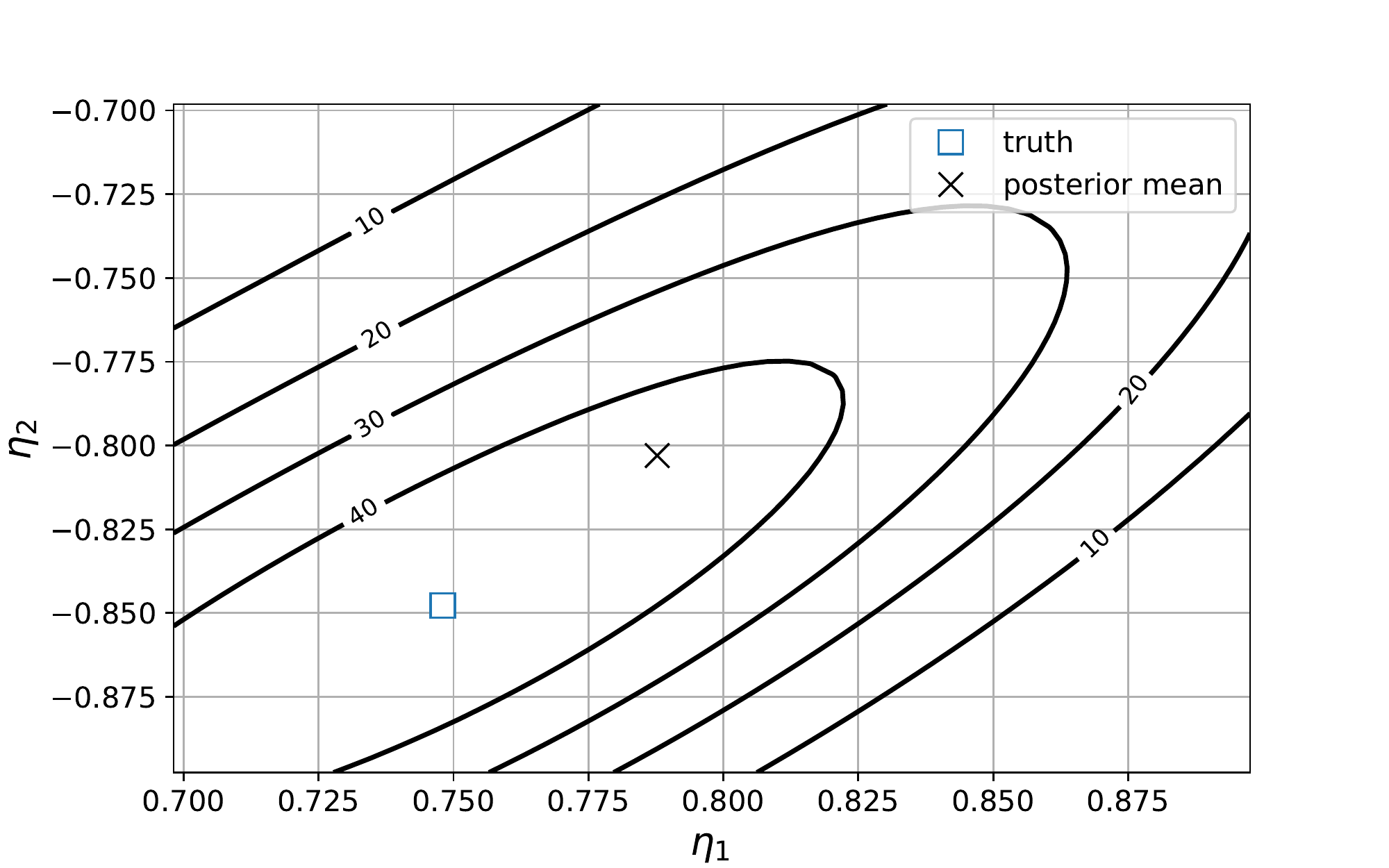}
		\caption{$~$}
	\end{subfigure}
	\begin{subfigure}[b]{0.49\textwidth}
		\centering
		\includegraphics[scale=0.35]{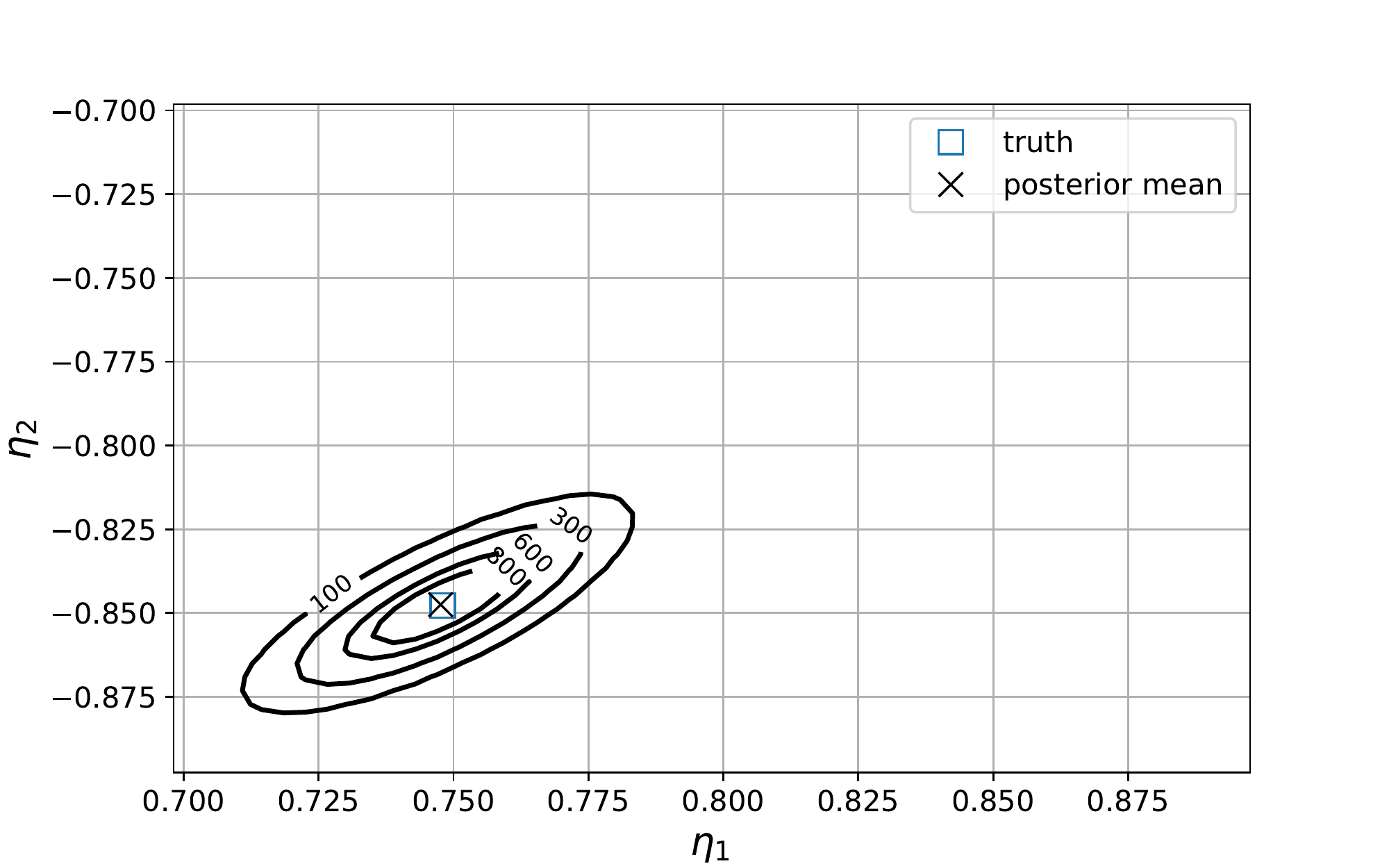}
		\caption{$~$}
	\end{subfigure}
	\caption{Contour plots of the posterior densities and their mean values for (a) initial DOE and (b) A-optimal DOE.
		The ground truth value is $\eta_1 =  0.748$, $\eta_2 =-0.848$.}
	\label{fig:eit_posteriors}
\end{figure}

\section{Conclusion}\label{sec:conclusion}

This study presents an efficient computational method for estimating the tECV within the A-optimal DOE framework using the PACE approach.
The intractability of the posterior distribution does not affect the computational efficiency of our approach,  as both the approximating and the sampling the posterior distribution are excluded.
Moreover, we derive an asymptotic error estimation of our method and verify it through numerical experiments.

To address continuous design domains, we combine the PACE framework with stochastic optimization algorithms to seek the A-optimal DOE.
We demonstrate our method by using the ANN to approximate the CE.
The stochastic optimization algorithms used for optimizing the ANN's weights and for finding the optimal design are integrated, as their loss functions are identical.
 We further propose a nonlocal approximation of the CE to reduce the number of evaluations of the observation map, which can be computationally demanding in practice. Numerical experiments show that our approach requires significantly fewer evaluations of the observational map compared to the crude IS method.

\appendix

\section{Computing tECV using the IS approach}
\label{appendix:ll-approach}

This appendix describes the computation of the ECV using the IS estimator of CE.
This approach requires a double-loop MC simulation.
The inner one estimates the posterior variances using Eq.~(\ref{eq:posterior_variance}),
while the outer one estimates the expected posterior variance  using Eq.~(\ref{eq:totalVar}).

The outer loop for a given design setting $\vd$ is formulated as
\begin{equation}
	\widehat{V}^{\text{IS}} = \dfrac{1}{N_{\text{o}}} \sum_{i=1}^{N_{\text{o}}} \sum_{k=1}^{n}\widehat{\operatorname{Var}} (Q\mid \vy^{(i)}),
	\quad \text{with} \quad \vy^{(i)} = h(\vq^{(i)}, \vd) + \xi^{(i)}, \; i =1, \dots, N_{\text{o}},
\end{equation}
where $\{\vq^{(i)}\}_{1}^{N_{\text{o}}}$ and $\{\xi^{(i)}\}_{1}^{N_{\text{o}}}$ are the $N_{\text{o}}$ \emph{i.i.d.}
samples of the RV $\rvQ$ and $\rvXi$, respectively.
For each sample pair $(\vq^{(i)}, \vy^{(i)})$, an inner loop is performed to estimate the posterior variance as $\widehat{\operatorname{Var}} (Q\mid \vy^{(i)})$
\begin{equation}
	\widehat{\operatorname{Var}} \left[\rvQ\mid \vy^{(i)}\right] =
	 \dfrac{\sum_{j=1}^{N_{\text{i}}} [\vq^{(i, j)}]^{\odot 2} \pi_{\rvXi} \left ( h(\vq^{(i, j)}, \vd) -  \vy^{(i)}\right)}
	{\sum_{j=1}^{N_{\text{i}}}\pi_{\rvXi} \left (   h(\vq^{(i, j)}) -  \vy^{(i)}\right ) } 
	-  \Biggl[ \dfrac{\sum_{j=1}^{N_{\text{i}}} \vq^{(i, j)} \pi_{\rvXi} \left ( h(\vq^{(i, j)}, \vd) -  \vy^{(i)}\right)}
	{\sum_{j=1}^{N_{\text{i}}}\pi_{\rvXi} \left ( h(\vq^{(i, j)}, \vd) -  \vy^{(i)}\right ) } \Biggr]^{\odot 2} ,
\end{equation}
where $\{\vq^{(i, j)}\}_{i=1, j = 1}^{N_{\text{o}}, N_{\text{i}}}$ are $N_{\text{i}} \times N_{\text{o}}$ \emph{i.i.d.} samples of the RV $\rvQ$.
This method requires $(N_{\text{i}}+1)\times N_{\text{o}}$ evaluations of the observational map $h$.

\section{Proofs}

\subsection{Proof of Proposition~\ref{proposition:tecv_via_conditional_expectation}}\label{appendix:proof_totalVar}
\begin{proof}
	By combining the laws of total mean  and total variance of the conditional expectation, \emph{i.e.}
	\begin{equation}
		\begin{aligned}
			\expectation{\expectation{\rvQ_i \mid {\rvY_{\vd}}}} &= \expectation{\rvQ_i},\\
			 \variance{\rvQ_i}  &= \expectation{\variance{\rvQ_i \mid {\rvY_{\vd}}}} + \variance{\expectation{\rvQ_i \mid {\rvY_{\vd}}}}
		\end{aligned}
	\end{equation}
	for $i=1, \dots, \dimq$, we attain
	\begin{align}
		\expectation{\variance{\rvQ_i|{\rvY_{\vd}}}} &= \variance{\rvQ_i}  -  \variance{\expectation{\rvQ_i \mid {\rvY_{\vd}}}} \\
		&= \expectation{\rvQ_i^2
			- \expectation{\rvQ_i}^2}
		-\expectation{\expectation{\rvQ_i \mid {\rvY_{\vd}}}^2 - \bigl(\expectation{\expectation{\rvQ_i \mid {\rvY_{\vd}}}}\bigr ){}^2}\\
		&=\expectation{\rvQ_i^2- \expectation{\rvQ_i \mid {\rvY_{\vd}}}^2}.
	\end{align}
	Hence, the tECV can be formulated as
	\begin{align}
		V(\vd)
		&\equiv \sum_{i=1}^\dimq \expectation{\variance{\rvQ_i|{\rvY_{\vd}}}}\\
		& = \expectation{ \sum_{i=1}^\dimq \rvQ_i^2 -\expectation{\rvQ_i \mid {\rvY_{\vd}}}^2}\\
		& = \expectation{ \left \lVert \rvQ -\expectation{\rvQ \mid {\rvY_{\vd}}}\right \rVert_{2}^2}.
	\end{align}
\end{proof}

\subsection{Proof of Proposition~\ref{proposition:error_estimation}}\label{appendix:proof_error_estimation}

\begin{proof}
	We structure our proof of Proposition \ref{proposition:error_estimation} in three steps.
	
	\emph{Step 1: Estimation of the optimization error $\epsilon_{\text{opt}}$.}
	Following the central limit theorem, \\
	$\sqrt{N}\bigl (\widehat{\mathcal{M}}(f^*\vert D_N)
	- \mathcal{M}(f^*) \bigr)\rightarrow \mathcal {N}(0, \variance { \left \lVert \rvQ - f^*(\rvY_{\vd})\right \rVert_{2}^2})$ as $N  \gg 1$.
	For a RV $X\sim \mathcal{N}(0, \sigma^2)$,
	we have $\expectation{\vert X \vert}= \sqrt{2/\pi}\;\sigma$.
	Consequently, the statistical error of the MC estimator
	$\widehat{\mathcal{M}}(f^*\vert D_N)$ can be asymptotically estimated as
	\begin{equation}\label{eq:e_MC_N}
		\begin{aligned}
			\expectation{
				\left \lvert \widehat{\mathcal{M}}(f^* \mid D_N) - \mathcal{M}(f^*) \right \rvert} &\approx
			\sqrt{\dfrac{2}{\pi N}} \variance {\left \lVert \rvQ - f^*(\rvY_{\vd}) \right\rVert_{2}^2}^{1/2}\\
			&=\mathcal{O}\Bigl(\dfrac{2}{\sqrt{\pi N}} \expectation {\left \lVert \rvQ - f^*(\rvY_{\vd}) \right \rVert_{2}^2}\Bigr),
		\end{aligned}
	\end{equation}
	where we use Assumption \ref{assumption2} to obtain the last step.
	Combining Assumption \ref{assumption3} and the result stated in Eq.~(\ref{eq:e_MC_N}) yields
	the following estimation of the optimization error $\epsilon_{\text{opt}}$
	\begin{equation}\label{eq:opt_error_estimation}
		\epsilon_{\text{opt}} = \mathcal{O}\Bigl(\dfrac{2}{\sqrt{\pi N}} \expectation {\left \lVert \rvQ
			- f^*(\rvY_{\vd}) \right \rVert_{2}^2}\Bigr).
	\end{equation}

	Moreover, using the orthogonal property $\expectation{(\rvQ - f^{*}({\rvY_{\vd}}))^\top f({\rvY_{\vd}}))}=0$
	for every $f \in \mathcal{S'}$, we obtain
	
	\begin{equation}\label{eq:equivalent_opt_error}
		\begin{aligned}
			\epsilon_{\text{opt}}  &= \expectation{\left \lVert \mathcal{M}(f(\cdot;\weight_{D_N}))
				- \mathcal{M}(f^*)  \right \vert} \\
			&= \expectation{\Biggl \vert
				\expectation{\left \lVert \rvQ - f({\rvY_{\vd}}; \weight_{D_N})\right \rVert_{2}^2
					\;  \Big \vert  \; \weight_{D_N}}
				- \expectation{\left \lVert \rvQ - f^*({\rvY_{\vd}})\right \rVert_{2}^2} \Biggr \vert}\\
			&= \expectation{\expectation {\left\lVert f({\rvY_{\vd}}; \weight_{D_N}) - f^*({\rvY_{\vd}}) \right \rVert_{2}^2 \;  \Big \vert  \; \weight_{D_N}}}\\
			&= \expectation{\left \lVert f({\rvY_{\vd}}; \weight_{D_N}) - f^*({\rvY_{\vd}})\right \rVert_{2}^2}.
		\end{aligned}
	\end{equation}
	
	By combining Eqs.~(\ref{eq:opt_error_estimation}) and (\ref{eq:equivalent_opt_error}), we obtain
	\begin{equation}\label{eq:f_err_norm}
		\expectation{\left \lVert f({\rvY_{\vd}}; \weight_{D_N}) - f^*({\rvY_{\vd}})\right \rVert_{2}^2}
		= \mathcal{O}\Bigl(\dfrac{2}{\sqrt{\pi N}} \expectation { \left \lVert \rvQ - f^*(\rvY_{\vd})\right \rVert_{2}^2}\Bigr).
	\end{equation}
	
	\emph{Step 2: Estimation of the MC estimator error $\epsilon_{\text{MC}}$.}
	Following the central limit theorem, for fixed hyperparameters
	$\weight_{D_N}$ as $M \gg 1$, we obtain
	\begin{equation}
	\sqrt{M} \left(
	\widehat{\mathcal{M}}(f(\cdot;\weight_{D_N}) \; \big \vert \;  D_M)
	- \mathcal{M}\left (f({\rvY_{\vd}}; \weight_{D_N}\right) \right)
	\rightarrow \mathcal {N} (0, \variance{\left\lVert \rvQ -
		f({\rvY_{\vd}}; \weight_{D_N}) \right \rVert_{2}^2
		\; \Big \vert \;
		\weight_{D_N}}^{1/2}).
	\end{equation}
	We estimate
	the error $\epsilon_{\text{MC}}$ as $M \to \infty$ using the central limit theorem as
	\begin{equation}\label{eq:MC_error_estimation}
		\begin{aligned}
			\epsilon_{\text{MC}} &\equiv  \expectation {\expectation {\bigl \vert \widehat{\mathcal{M}}(f(\cdot;\weight_{D_N})|D_M)
					- \mathcal{M}(f({\rvY_{\vd}}; \weight_{D_N})) \bigr\vert
					\; \Big \vert \; \weight_{D_N} }}\\
			&\approx \expectation{\sqrt{\dfrac{2}{\pi M}} \variance{\left \lVert \rvQ -
					f({\rvY_{\vd}}; \weight_{D_N}) \right \rVert_{2}^2\; \Big \vert \;
					\weight_{D_N}}^{1/2} } \\
			&=\mathcal{O}\Bigl(\sqrt{\dfrac{2}{\pi M}} \variance{\left \lVert \rvQ -
				f^*({\rvY_{\vd}})\right \rVert_{2}^2}^{1/2}
			\Bigr) \\
			&=\mathcal{O}\left (\dfrac{2}{\sqrt{\pi M}}
			\expectation { \left \lVert \rvQ - f^*(\rvY_{\vd})\right \rVert_{2}^2} \right),
		\end{aligned}
	\end{equation}
	where the third step is obtained owing to $\expectation{\left \lVert f({\rvY_{\vd}};\weight_{D_N})
		- f^* ({\rvY_{\vd}})  \right \rVert_{2}^2} = \mathcal{O}(2/\sqrt{\pi N})$ as stated in Eq.~(\ref{eq:f_err_norm}).
	
	\emph{Step 3.} Eventually, by combining Eqs. (\ref{eq:error_bound}), (\ref{eq:opt_error_estimation}), and (\ref{eq:MC_error_estimation}), we can estimate the error of the estimator $\widehat{V}_{\vd}(D_N, D_M)$ as
	\begin{equation}
		\expectation{\left \lvert \widehat{\mathcal{M}}(f(\cdot;\weight_{D_N})\; \big \vert \; D_M)
			- \mathcal{M}(\phi_d) \right \rvert} = \mathcal{O}\Biggl(\Bigl(\dfrac{2}{\sqrt{\pi N}} + \dfrac{2}{\sqrt{\pi M}}\Bigr) \expectation { \left \lVert \rvQ - f^*(\rvY_{\vd})\right \rVert_{2}^2}\Biggr) + \epsilon_{S'}.
	\end{equation}
\end{proof}

\subsection{Orthogonality property of the CE}

\label{appendix:orthogonal_projection}
\begin{theorem}[Orthogonality property]\label{theo:orthogonal_projection}
	Let $\rvQ$ and $\rvY$ be the finite-variance RVs valued in $ \mathbb{R}^\dimq$ and  $\mathbb{R}^\dimy$, respectively .
	Let $L_2(\sigma_{\rvY})$ be the collection of all random variables of type $g(\rvY)$,
	 where $g: \mathbb{R}^\dimy \rightarrow \mathbb{R}^\dimq$
	  is an arbitrary function satisfying $\expectation{\lVert g(\rvY)\rVert_{2}^2} < \infty$. Then, 
	   \begin{itemize}
	   	\item  [A)]
		\begin{equation}\label{eq:orthogonal}
			\expectation{\; f (\rvY){}^\top \;\bigl(\rvQ - \expectation{\rvQ \mid \rvY}\bigr)\; } =0,
		\end{equation}
	  \item [B)] the CE $\expectation{\rvQ \mid \rvY}$ is the RV in $L_2(\sigma_{\rvY})$
	  that minimizes the MSE, i.e.,
		\begin{equation}
			\begin{aligned}
				\expectation{\rvQ \mid  \rvY} &= \phi  (\rvY),\\
				\text{where} \quad
				\phi  &=  \arg \min_{f (\rvY) \in L_2(\sigma_{\rvY})} \expectation{\left \lVert \rvQ- g(\rvY) \right \rVert_{2}^2}.
			\end{aligned}
		\end{equation}
	\end{itemize}
\end{theorem}
\begin{proof}
	For any measurable set
	$B\subset \sR^\dimy$, let $A= \rvY^{f \; -1} (B)\equiv \{ \omega\; : \rvY(\omega) \in B\}$.
	 As $A\in \sigma_{\rvY}$, we achieve
	\begin{equation}\label{eq:orthogonal_simpleFunct}
	\begin{aligned}
	\expectation{\mathbf{1}_B(\rvY) \;\left(\rvQ - \expectation{\rvQ \mid \rvY}\right)} &=\expectation{\mathbf{1}_B (\rvY) \; \rvQ}- \expectation{\mathbf{1}_B(\rvY)\expectation{\rvQ \mid \rvY}}\\
	&= \int_A \rvQ(\omega)\dint \probability (\omega)- \int_A \expectation{\rvQ | \rvY}(\omega) \dint \probability (\omega) \\
	&=0,
	\end{aligned}
	\end{equation}
	where $\mathbf{1}_B(Y(\omega)) = 1$ if $Y(\omega) \in B$ and $0$ otherwise.

	Let $\rvX^\dimq$ be a $\sigma_{\rvY}$-measurable simple RV that converges in terms of distribution to $g\circ\rvY$. Using the result from Eq.~(\ref{eq:orthogonal_simpleFunct}), we prove part A) ot Theorem \ref{theo:orthogonal_projection}.

	Let $Z= \expectation{\rvQ \mid \rvY}- \phi(\rvY)$, we obtain
	\begin{equation}\label{eq:orthogonal_projection_proof}
	\begin{aligned}
	\expectation{\bigl \lVert \rvQ- \phi(\rvY)\bigr\rVert_{2}^2}&= \expectation{\bigl\lVert \rvQ- \operatorname{E}{\bigl[\rvQ \mid \rvY\bigr]}+ Z\bigr\lVert ^2}\\
	&= \expectation{ \bigl \lVert \rvQ- \operatorname{E}{\bigl[\rvQ \mid \rvY\bigr]}\bigr\rVert_{2}^2}+ \expectation{\bigl \lVert Z\bigr\rVert_{2}^2},\\
	\end{aligned}
	\end{equation}
	as the cross-product term vanishes. Eq.~(\ref{eq:orthogonal_projection_proof}) directly shows that \\
	$\expectation{\bigl\lVert \rvQ- \phi(\rvY)\bigr\rVert_{2}^2}$ is minimized when $Z$ is
	 the $n$-dimensional vector of zeros,
	 or $ \expectation{\rvQ \mid \rvY} = \phi(\rvY)$ (part B).
\end{proof}

\section{Linear approximation of the CE}
\subsection{Analytical formulation of the linear approximation}
\label{appendix:linear_approximation}
The linear approximation of the CE $f_{\text{l}}(\rvY_{\vd}) = \mat{A} \rvY_{\vd} +\vek{b}$ is the solution to the least mean-squares problem
\begin{equation}
\mat{A}, \vek{b} = \quad \arg \min_{\mat{A}' \in \sR^{\dimq\times \dimy}, \vek{b}' \in \sR^\dimq}
 \expectation{\bigl\lVert\rvQ-  \mat{A}' \rvY_{\vd}  -\vek{b}'\bigr\rVert_{2}^2}.
\end{equation}
Using the first order necessary conditions
\begin{equation}
\begin{aligned}
\expectation{\rvQ-  \mat{A}\rvY_{\vd} -\vek{b}} &=\vek{0}_{\dimq},\\
\expectation{(\rvQ-  \mat{A}\rvY_{\vd} -\vek{b}) \rvY_{\vd}^\top }&=\mat{0}_{\dimq\times \dimy},
\end{aligned}
\end{equation}
where $\mat{0}_{\dimq \times \dimy}$ denotes the $\dimq \times \dimy$ matrix of zeros, we obtain
\begin{equation}
\begin{aligned}
\mat{A} &=\expectation{\bigl(\rvQ-\operatorname{E}[\rvQ]\bigr)\rvY_{\vd}^{\top}} \expectation{\bigl(\rvY-\operatorname{E}{[\rvY_{\vd}]}\bigr)\rvY^{\top}}^{-1}\\
&=\expectation{\bigl(\rvQ-\operatorname{E}{[\rvQ]} \bigr) {\bigl(\rvY_{\vd}-\operatorname{E}{[\rvY_{\vd}]}\bigr)^{\top}}} \expectation{\bigl(\rvY_{\vd}-\operatorname{E}{[\rvY_{\vd}]}\bigr) {\bigl(\rvY-\operatorname{E}{[\rvY_{\vd}]}\bigr)^{\top}}}^{-1} \\
&=\cov{\rvQ,\rvY_{\vd}}\cov{\rvY_{\vd}}^{-1}, \\
\vek{b} &= \expectation{\rvQ-  \mat{A}\rvY_{\vd}}.
\end{aligned}
\end{equation}

\subsection{Empirical linear approximation of the CE}\label{appendix:linear_approximation_empirical}
Given dataset $D_N = {(\vq^{(i)}, \vy^{(i)})}_{i=1}^N$ of the $N$ \emph{i.i.d.} samples of the RV pair $(\rvQ, \rvY)$,
 the empirical linear approximation of the CE is obtained as
 \begin{equation}
 \widehat{\phi} (\vy) = \widehat{\operatorname{Cov}}[{\rvQ,\rvY_{\vd}}]\;\bigl [\widehat{\operatorname{Cov}}[{\rvY_{\vd}}]\bigr ]^{-1} \vy + \widehat{\vek{b}}
 \end{equation}
where
 \begin{equation}\label{eq:empirical_linear_coefficients}
 \begin{aligned}
    & \widehat{\operatorname{Cov}}[{\rvQ,\rvY_{\vd}}]  = \dfrac{1}{N} \sum_{i=1}^{N} [\vq^{(i)} -
    \overline{\vq}]\,[\vy^{(i)}-\overline{\vy}]^\top, \\
    & \widehat{\operatorname{Cov}}[{\rvY_{\vd},\rvY_{\vd}}]  = \dfrac{1}{N} \sum_{i=1}^{N}
    [\vy^{(i)}-\overline{\vy}]\, [\vy^{(i)}-\overline{\vy}]^\top, \\
    &  \widehat{\vek{b}} = \overline{\vq} -
    \dfrac{1}{N} \sum_{i=1}^{N} \widehat{\operatorname{Cov}}[{\rvQ,\rvY_{\vd}}]\;
    \bigl[\widehat{\operatorname{Cov}}[{\rvY_{\vd}}]\bigr]^{-1} \vy^{(i)}.
  \end{aligned}
 \end{equation}
 Here, $\overline{\vq}$ and $\overline{\vy}$ are empirical means of the RVs $\rvQ$ and $\rvY_{\vd}$, respectively,
 computed as
 \begin{equation}
\overline{\vq}= \dfrac{1}{N} \sum_{i=1}^{N} \vq^{(i)}, \quad \overline{\vy} = \dfrac{1}{N} \sum_{i=1}^{N} \vy^{(i)}.
 \end{equation}
The reduced variance estimators corresponding to those in Eq.~(\ref{eq:empirical_linear_coefficients})
are straightforwardly obtained by replacing the role of the dataset $D_N$ with its augmented version $D_N^{\textnormal{vr}}$ (see
 (Eq.~\ref{eq:data_augmented_sets})).

\section{Reduced variance estimators of MSE and tECV}
\label{appendix:data_augmenting}

We show that the estimator
$\widehat{\mathcal{M}}^{\text{vr}}$ (see Eq.~(\ref{eq:mc_vd}))
 provides an estimation
 with reduced variance compared with the crude MC estimator $\widehat{\mathcal{M}}$ (see
  Eq.~(\ref{eq:crude})).
For a given vector $\vd$, we have
\begin{equation}\label{eq:MC_augmenteddata}
\begin{aligned}
\lim_{a \rightarrow \infty}\widehat{\mathcal{M}}^{\text{vr}} (f \mid D_N)
	&\;=_{\text{a.s.}}\;\dfrac{1}{N} \sum_{i=1,\dots, N}
	\operatorname{E} \Bigl[\bigl \lVert \vq^{(i)} - f \bigl(h(\vq^{(i)}, \vd) + \rvXi\bigr) \rVert_{2}^2 \Bigr] \\
    &=\;\dfrac{1}{N} \sum_{i=1,\dots, N}   \expectation{A(\rvQ, \rvXi)| \rvQ = \vq^{(i)}},
\end{aligned}
\end{equation}
$\probability$-almost surely, where
\begin{equation}
A(\vq, \vek{\xi}) \equiv \bigl \lVert \vq - f\bigl(h(\vq, \vd) + \vek{\xi}\bigr) \bigr\rVert_{2}^2.
\end{equation}
Let $\widehat{\mathcal{M}}^{\text{vr}*}(f)$ denote the right-hand-side term in Eq.~(\ref{eq:MC_augmenteddata}).
Because $\{q^{(i)}\}_{i =1}^N$ are \emph{i.i.d.}
 samples,
we approximately quantify the statistical errors of the estimators
$\widehat{\mathcal{M}}$
and $\widehat{\mathcal{M}}^{\text{vr}*}(f)$, respectively, as
\begin{subequations}
	\begin{align}
\variance{\widehat{\mathcal{M}}(f\vert D_N)- \expectation{A(\rvQ, \rvXi)} } &\approx
 \dfrac{ \variance{A(\rvQ, \rvXi) }}{N}, \\
\variance{\widehat{\mathcal{M}}^{\text{vr}*}(f)- \expectation{A
(\rvQ, \rvXi)} } &\approx
 \dfrac{ \variance{ \expectation{A(\rvQ, \rvXi) \mid \rvQ}}}{N}.
	\end{align}
\end{subequations}
Using the law of total variance,
\begin{equation}
\variance{ A(\rvQ, \rvXi)} =
\variance{ \expectation{A(\rvQ, \rvXi) \mid \rvQ}}
 + \expectation{\variance{A(\rvQ, \rvXi) \mid \rvQ}},
\end{equation}
we obtain
\begin{equation}
\variance{ \expectation{A(\rvQ, \rvXi) \mid \rvQ}}
\leq  \variance{ A(\rvQ, \rvXi)}.
\end{equation}
Therefore, for $a\gg 1$, $\variance{\widehat{\mathcal{M}}^{\text{vr}}(f \mid  D_N)
- \expectation{A(\rvQ, \rvXi)}}
 \leq \variance{\widehat{\mathcal{M}}(f \mid D_N)- \expectation{A(\rvQ, \rvXi)}}$.

\section{Gradient of the weighted MSE w.r.t. design parameters}
\label{appendix:gradients}
  The gradient
   $\nabla_{\vd} \Bigl[ \left\lVert \vq - f_{\text{N}}(h(\vq)+\xi,\vd) \right \rVert_{2}^2\Bigr]$ is expanded using the chain rule as
    \begin{equation}
    \begin{aligned}
        \nabla_{\vd}
		     \Bigl[\left \lVert q - f_{\text{N}}(h(\vq, \vd) +\xi,\vd) \right \rVert_{2}^2\Bigr]
		     = -2 & \Bigl[ \nabla_{\vd} \, f_{\text{N}}(\vy,\vd) +
		            \nabla_{\vy} \, f_{\text{N}}(\vy,\vd) \, \nabla_{\vd} \, h(\vq,\vd) \Bigl]^\top\\
		            & \Bigl[q - f_{\text{N}}(h(\vq,\vd) +\xi,\vd) \Bigr],
        \end{aligned}
    \end{equation}
where $\vy = h(\vq, \vd) +\xi$.
    Notably, for ANN $f_{\text{N}}$, the Jacobian matrices $\nabla_\vd f_{\text{N}}$ and $\nabla_\vy f_{\text{N}}$ can be directly obtained numerically
     using the backpropagation \cite{GrieWal08}.
    We finally obtain the value of
    $\nabla_{\vd} \Bigl[ \left \lVert \vq - f_{\text{N}}(h(\vq)+\xi,\vd) \right \rVert_{2}^2\Bigr](q^{(i)}, \vy^{(i,j)} ,\vd_0)$ as
    \begin{equation}
        \begin{aligned}
        \nabla_{\vd}
		     \Bigl[\left \lVert q - f_{\text{N}}(h(\vq, \vd) +\xi,\vd) \right \rVert_{2}^2\Bigr]&(\vq^{(i)}, \vy^{(i,j)},\vd_0)
		     = \\
		        -2& \Bigl[ \nabla_{\vd} \, f_{\text{N}}(\vy^{(i,j)},\vd_0)  + 
		        \phantom{\Bigl(} \nabla_{\vy} \, f_{\text{N}}(\vy^{(i,j)},\vd_0) \,  \nabla_{\vd} \,h(\vq^{(i)},\vd_0)  \Bigr]^\top\\
		        & \Bigl[\vq^{(i)} - f_{\text{N}}(\vy^{(i,j)},\vd_0) \Bigr].
		\end{aligned}
    \end{equation}

\section{ Derivation of the finite element formulation}\label{appendix:fem}

We define $\mathcal{H}:= H^1({B}\times\mathbb{R}^{N_{\rm{el}}})$ as the space of the solution for the potential field $(u(\omega),{U}(\omega))$ for $\omega\in\Omega$ and the Bochner space $L_{\mathbb{P}}^2(\Omega;\mathcal{H})$ as
\begin{equation}
  L_{\mathbb{P}}^2(\Omega;\mathcal{H}):= \left\{(u,{U}):\Omega\to\mathcal{H} \quad \text{s.t.}\, \int_{\Omega}\lVert (u(\omega),{U}(\omega))\rVert_{\mathcal{H}}^2\dint{}\mathbb{P}(\omega)<\infty\right\},
\end{equation}
where $U (\omega)=[U_1 (\omega), \dots, U_{N_{\rm{el}}}(\omega)]^\top$.
For the bilinear form $\mathcal{B}:\mathcal{H}\times\mathcal{H}\to\sR$, which is given as
\begin{equation}
  \mathcal{B}((u,{U}),(v,{V})):= \int_{\mathcal{D}} \sigma \, \nabla u \cdot \nabla v\dint{}{B}+\sum_{l=1}^{N_{\rm{el}}}\frac{1}{z_l}\int_{E_l}({U}_l-u)({V}_l-v)\dint{}E_l,
\end{equation}
 we aim to determine $(u,{U})\in L_{\mathbb{P}}^2(\Omega;\mathcal{H})$ such that the weak formulation
\begin{equation}
  	\mathcal{B}\left((u(\omega),{U}(\omega)),(v,{V})\right)= \sum_{l=1}^{N_{\rm{el}}} I_l \, U_l(\omega)
\end{equation}
is fulfilled for all $(v,{V})\in L_{\mathbb{P}}^2(\Omega;\mathcal{H})$ almost surely.

\bibliography{bibliography.bib}
\bibliographystyle{unsrt}
\end{document}